\documentclass[12pt]{amsart}

\usepackage[foot]{amsaddr}
\usepackage[toc]{appendix} 
\usepackage{latexsym}
\usepackage{amsfonts}
\usepackage{amssymb}
\usepackage{amsmath}
\usepackage{mathrsfs}
\usepackage{bbm}
\usepackage{dsfont}
\usepackage{hyperref}
\allowdisplaybreaks

\usepackage[toc]{appendix}

\newcommand{\be}{\begin{equation}}
\newcommand{\ee}{\end{equation}}
\newcommand{\bea}{\begin{eqnarray}}
\newcommand{\eea}{\end{eqnarray}}
\newcommand{\bean}{\begin{eqnarray*}}
\newcommand{\eean}{\end{eqnarray*}}
\newcommand{\brray}{\begin{array}}
\newcommand{\erray}{\end{array}}

\newcommand{\bdfn}{\begin{dfn}\rm}
\newcommand{\bthm}{\begin{thm}}
\newcommand{\blmma}{\begin{lmma}}
\newcommand{\bppsn}{\begin{ppsn}}
\newcommand{\bcrlre}{\begin{crlre}}
\newcommand{\bxmpl}{\begin{xmpl}}
\newcommand{\brmrk}{\begin{rmrk}\rm}

\newcommand{\edfn}{\end{dfn}}
\newcommand{\ethm}{\end{thm}}
\newcommand{\elmma}{\end{lmma}}
\newcommand{\eppsn}{\end{ppsn}}
\newcommand{\ecrlre}{\end{crlre}}
\newcommand{\exmpl}{\end{xmpl}}
\newcommand{\ermrk}{\end{rmrk}}

\newcommand{\bbc}{\mathbb{C}}

\newcommand{\bbn}{\mathbb{N}}
\newcommand{\bbr}{\mathbb{R}}

\newcommand{\bbt}{\mathbb{T}}

\newcommand{\clh}{\mathcal{H}}

\makeatletter
\let\@wraptoccontribs\wraptoccontribs
\makeatother
\title{Multiparameter Decomposable product systems}
\author{C.H. Namitha  and S. Sundar}
\address{Institute of Mathematical Sciences, A CI of Homi Bhabha National Institute, 4th cross street, CIT Campus, Taramani, Chennai, India, 600113}
\email{namithachanguli7@gmail.com, sundarsobers@gmail.com}

\usepackage{enumerate}
\usepackage{bbm}
\usepackage{cleveref}

\newtheorem{definition}{Definition}[section]
\newtheorem{prop}[definition]{Proposition}
\newtheorem{theorem}[definition]{Theorem}
\newtheorem{corollary}[definition]{Corollary}
\newtheorem{lemma}[definition]{Lemma}
\newtheorem{remark}[definition]{Remark}
\numberwithin{equation}{section}
\newcommand{\RNum}[1]{\uppercase\expandafter{\romannumeral #1\relax}}
\newcommand{\R}{\mathbb{R}}
\newcommand{\RH}{\frac{\mathbb{R}^{d}}{H}}
\newcommand{\RO}{\frac{\mathbb{R}^{d-1}}{H_{0}}}

\begin{document}
\maketitle

\begin{abstract}
 In \cite{Arv_Path},  Arveson proved that a $1$-parameter decomposable product system is isomorphic to the product system
of a CCR flow. We show that the structure of a generic decomposable product system, over higher dimensional cones, modulo twists by multipliers, is given by an isometric representation $V$ of the cone and a certain $2$-cocycle for $V$. 
Moreover, we compute the space of $2$-cocycles for shift semigroups associated to transitive actions of a higher dimensional cone.
\end{abstract}

\noindent {\bf AMS Classification No. :} {Primary 46L55; Secondary 46L99.}  \\
{\textbf{Keywords :}} $E_0$-semigroups, Decomposable Product Systems, Cocycles.

%\tableofcontents

\section{Introduction}
Product systems, over $(0,\infty)$, were introduced by Arveson (\cite{Arv_Fock}) as a complete invariant to classify $1$-parameter $E_0$-semigroups whose study has been of interest to  many authors.  Powers (\cite{Powers_Index}, \cite{Powers_TypeIII}, \cite{Powers}, \cite{Powers_CPflow}) and Arveson (\cite{Arv_Fock}, \cite{Arv_Fock2}, \cite{Arv_Fock3}, \cite{Arv_Fock4}) are considered to be the founders of the subject of $E_0$-semigroups with Tsirelson (\cite{Tsirelson}, \cite{Tsi}, \cite{Tsirelson_gauge}) being another major contributor who with his probabilistic ideas made significant breakthroughs in the $1$-parameter theory of $E_0$-semigroups.  Within the framework of noncommutative 
dynamics, the role of product systems over more general semigroups  was made clear by Shalit  (\cite{Shalit}, \cite{Shalit_2008}, \cite{Shalit_2010} and \cite{Shalit_2011}), by
Solel (\cite{Solel}) and by Shalit and Solel (\cite{Shalit_Solel}) who investigated, in detail,  the dilation theory of CP-semigroups to $E_0$-semigroups for semigroups other than $(0,\infty)$, and demonstrated that it is worth  studying product systems over general semigroups.  More recently, an intrinsic study of  multiparameter $E_0$-semigroups, especially over cones, and  product 
systems  was undertaken in \cite{Anbu_Sundar},  \cite{Anbu_Vasanth}, \cite{murugan}, and \cite{SUNDAR}. 
This paper is an attempt in this direction. Here, we  complete the analysis of decomposable product systems over cones initially dealt with in \cite{SUNDAR}. 

Let $P$ be a closed convex cone in $\bbr^d$. Without loss of generality, we can assume that $P-P=\mathbb{R}^{d}$. We also make the assumption  that  $P \cap -P=\{0\}$. A product system over $P$ is a measurable field of Hilbert spaces,
over the base space $P$, endowed with a product that is associative and that is compatible with the measurable structure. 
It follows from the seminal work of Arveson (\cite{Arv_Path}) that, in the one parameter case, decomposable product systems are precisely those that of CCR flows, and there are only countably 
many of them. 

Arveson's proof, in the one parameter case, relies on constructing an isometric representation out of a decomposable product system and then solving several cohomological issues that turned 
out to be trivial. 
Imitating Arveson,  the construction of an isometric representation from a  decomposable product system over higher dimensional cones, was achieved in \cite{SUNDAR}. In particular, given a decomposable product system $E$ over $P$, a pure isometric representation, which we denote by $V^{E}$, of $P$ was constructed. Let us call $V^{E}$ the isometric representation associated to $E$. It was  observed in \cite{SUNDAR} that in the presence of a unit,  cohomological problems do not arise and consequently, it was proved in \cite{SUNDAR} that if $E$ is a decomposable product system over $P$, then $E$ is isomorphic to the product system of the CCR flow $\alpha^{V}$ (where $V=V^{E})$ if and only if $E$ has a unit. Moreover, examples of decomposable product systems not having units were 
given in \cite{SUNDAR}. It was further demonstrated  in \cite{Poisson} that  inhomogeneous Poisson processes on $\mathbb{R}^d$ give rise to an abundant supply of such examples. 

The absence of units in the multiparameter case is a clear indication of the fact that there are non-trivial  cohomological obstructions in the multiparameter case. Thus, it is a natural question to understand these cohomological issues  in more detail and, if possible, to compute the cohomology groups in concrete situations. 
 In this paper, we continue the imitation  undertaken in \cite{SUNDAR}, spell out the cohomological issues in more detail, and complete the analysis by describing the structure of a generic decomposable product system over $P$. 

Next, we explain the results obtained. 
Let $V=\{V_a\}_{a \in P}$ be a strongly  continuous semigroup of isometries on a separable Hilbert space $\mathcal{H}$, i.e. 
\begin{enumerate}
\item[(1)] for $a \in P$, $V_a$ is an isometry on $\mathcal{H}$, 
\item[(2)] for $a,b \in P$, $V_aV_b=V_{a+b}$, and
\item[(3)] the map $P \ni a \to V_a\xi \in \mathcal{H}$ is continuous for every $\xi \in \mathcal{H}$. 
\end{enumerate}
We also call a strongly continuous semigroup of isometries indexed by $P$ an isometric representation of $P$.  We assume that $V$ is pure, i.e $\displaystyle \bigcap_{a \in P}Ran(V_a)=\{0\}$. 

Let $\Gamma:P \times P \to \mathcal{H}$ be a map. We say that $\Gamma$ is a $2$-cocycle for $V$ if 
\begin{enumerate}
\item[(1)] for $\xi \in \mathcal{H}$, the map $P \times P \ni (a,b) \to \langle \Gamma(a,b)|\xi \rangle \in \bbc$ is measurable,
\item[(2)] for $a,b \in P$, $\Gamma(a,b) \in Ker(V_b^*)$, and
\item[(3)] for $a,b,c \in P$, 
\[
\Gamma(a,b+c)+V_b\Gamma(b,c)=\Gamma(a,b)+V_b\Gamma(a+b,c).\]
\end{enumerate}

Let $\Gamma:P \times P \to \mathcal{H}$ be a $2$-cocycle. We say that $\Gamma$ is \emph{admissible} if there exists a measurable map $\alpha:P \times P \to \mathbb{T}$ such that 
\begin{equation}
\label{admissibleintro}
\frac{\alpha(a,b)\alpha(a+b,c)}{\alpha(a,b+c)\alpha(b,c)}=e^{i Im \langle \Gamma(a,b+c)|V_b\Gamma(b,c)\rangle}
\end{equation} for $a,b,c\in P$.

Let $V=\{V_a\}_{a \in P}$ be a pure semigroup of isometries on a separable Hilbert space $\mathcal{H}$. Suppose $\Gamma:P \times P \to \mathcal{H}$ is an admissible $2$-cocycle for $V$ and suppose $\alpha:P \times P \to \mathbb{T}$ is a measurable map that satisfies Eq. \ref{admissibleintro}. 
Let $E^V:=\{E^V(a)\}_{a \in P}$ be the product system of the CCR flow associated to $V$. We denote the product on $E^V$ by $\odot$.

For $a \in P$, set $E(a):=E^V(a)$. Denote the field $\{E(a)\}_{a \in P}$ of Hilbert spaces by $E$. Consider the following multiplication rule 
on $E$. For $S \in E(a)$ and $T \in E(b)$, define $S \cdot T \in E(a+b)$ by 
\[
S \cdot T=\alpha(a,b)S \odot W(\Gamma(a,b))T.\]
In the above formula, $\{W(\xi):\xi \in \mathcal{H}\}$ denotes the collection of Weyl operators. 
It is straightforward to verify that the above product makes $E$ a product system.   As the product system $E$ depends on the triple $(\alpha,\Gamma,V)$, we denote $E$ by $E^{(\alpha,\Gamma,V)}$. It is not difficult to verify that $E^{(\alpha,\Gamma,V)}$ is decomposable. 

Our main results are the following two theorems. 
\begin{theorem}
\label{mainintro}
Let $E$ be a decomposable product system over $P$. Denote the isometric representation associated to $E$, i.e $V^{E}$ by $V$. Then, there exist an admissible $2$-cocycle $\Gamma$ for $V$ and a measurable map $\alpha:P \times P \to \bbt$ satisfying Eq. \ref{admissibleintro} such that $E$ is isomorphic to $E^{(\alpha,\Gamma,V)}$. 
\end{theorem}

We need a bit of terminology before stating the next theorem. Let $E$ and $F$ be two product systems over $P$, and suppose $\displaystyle \theta=\{\theta_a\}_{a  \in P}: E \to F$ is a Borel map. We say that $\theta$ is a \emph{projective isomorphism} if there exists a Borel multiplier $\omega:P \times P \to \bbt$ such that 
\begin{enumerate}
\item[(1)] for $a \in P$, $\theta_a:E(a) \to F(a)$ is a unitary, and
\item[(2)] for $u \in E(a)$ and $v \in E(b)$, $\theta_{a+b}(uv)=\omega(a,b)\theta_a(u)\theta_b(v)$.
\end{enumerate}
Two product systems are said to be \emph{projectively isomorphic} if there exists a projective isomorphism between them. 
It is not difficult to verify that if $E$ and $F$ are two decomposable product systems that are projectively isomorphic, then $V^{E}$ and $V^{F}$ are unitarily equivalent. 

\begin{theorem}
\label{mainintro1}
The product systems $E^{(\alpha_1,\Gamma_1,V^{(1)})}$ and $E^{(\alpha_2,\Gamma_2,V^{(2)})}$ are projectively isomorphic if and only if there exist a unitary $U:\mathcal{H}_1\to \mathcal{H}_2$ and a measurable map $\xi:P \to \mathcal{H}_2$ such that 
\begin{enumerate}
\item[(1)] for  $a \in P$, $UV^{(1)}_{a}U^{*}=V^{(2)}_{a}$, 
\item[(2)] for  $a,b \in P$, $V^{(2)}_{a}(U\Gamma_1(a,b)-\Gamma_2(a,b))=\xi_{a+b}-\xi_{a}-V_{a}^{(2)}\xi_b$, and
\item[(3)] for $a \in P$, $\xi_{a} \in Ker(V_a^{(2)*})$. 
\end{enumerate}
Here, $\mathcal{H}_i$ is the Hilbert space on which $V^{(i)}$ acts for $i=1,2$.
\end{theorem}
Next, we cast Thm. \ref{mainintro} and Thm. \ref{mainintro1} in cohomological language as follows. Suppose $V:=\{V_a\}_{a \in P}$ is an isometric representation of $P$ that acts on the Hilbert space $\mathcal{H}$. Denote the space of $2$-cocycles and the set of admissible cocycles for $V$ by $Z^{2}(P,\mathcal{H})$ and by $Z_{a}^{2}(P,\mathcal{H})$ respectively. 
For $\Gamma \in Z^{2}(P,\mathcal{H})$, we say that $\Gamma$ is a coboundary if there exists a measurable map $\xi:P \to \mathcal{H}$ such that 
\begin{enumerate}
\item[(1)] for $a \in P$, $\xi_a \in Ker(V_a^*)$, and
\item[(2)] for $a,b \in P$, $V_a\Gamma(a,b)=\xi_{a+b}-\xi_a-V_a\xi_b$.
\end{enumerate}
Denote the space of coboundaries by $B^{2}(P,\mathcal{H})$ and the resulting cohomology group by $H^{2}(P,\mathcal{H})$. The image of $Z_{a}^{2}(P,\mathcal{H})$ in $H^{2}(P,\mathcal{H})$, under the natural quotient map, is denoted by $H^{2}_{a}(P,\mathcal{H})$.  To stress the dependence of $H^{2}_a(P,\mathcal{H})$ on $V$, we denote it by $H_{a}^{2}(P,\mathcal{H},V)$, or simply by $H_{a}^{2}(P,V)$. We use similar notation for $H^{2}(P,\mathcal{H})$ and  for others. 
Let \[M_V:=\{V_{a},V_{a}^{*}: a \in P\}^{'},\] and let $\mathcal{U}(M_V)$ be the unitary group of the von Neumann algebra $M_V$. Note that the group $\mathcal{U}(M_V)$ acts naturally on $H^{2}_{a}(P,\mathcal{H})$. 

Denote the collection (up to unitary equivalence) of pure isometric representations of $P$ by $Isom(P)$, and denote the collection (up to projective isomorphism) of decomposable product systems over $P$ by $\mathcal{D}(P)$. For $V \in Isom(P)$, let $\mathcal{D}_{V}(P)$ be the collection (up to projective isomorphism) of decomposable product systems whose associated isometric representation is $V$. The results of \cite{SUNDAR} assert that 
\begin{equation}
\label{disjoint intro}
\mathcal{D}(P)=\coprod_{V \in Isom(P)}\mathcal{D}_V(P).
\end{equation}
Thm. \ref{mainintro} and Thm. \ref{mainintro1} assert that, for a pure isometric representation $V$ of $P$ acting on a Hilbert space $\mathcal{H}$,  
\begin{equation}
\label{projective isomorphism equivalence intro}
\mathcal{D}_V(P)=\frac{H^{2}_a(P,\mathcal{H},V)}{\mathcal{U}(M_V)}.
\end{equation}

In view of the above two equations, we believe that `a complete classification or an enumeration' of decomposable product systems in the higher dimensional case,  just like in the $1$-parameter case, is  harder to achieve.  For, this relies on describing all pure semigroup of isometries of $P$, which is complicated. 
Even the irreducible ones, even for the case $P=\mathbb{R}_{+}^{2}$, is not known.  Thus, providing a  'complete list' of decomposable product systems in the higher dimensional case is beyond the scope of the authors. 

Having said the above, it is still desirable and is of some interest to compute the cohomology group $H^{2}(P,\mathcal{H},V)$ (and $H^2_a(P,\mathcal{H},V)$)  for a few examples. In the one parameter case, Theorem 5.3.2 of \cite{arveson} states that, for every pure isometric representation $V$ of $[0,\infty)$, the space of $2$-cocycles (modulo coboundaries) is zero. We show that this is not necessarily true by explicitly computing $H^2(P,V)$ (which turns out be non-zero) for the  shift semigroup associated to a transitive action of the cone $P$. It is this non-vanishing of the cohomology, in the multiparameter case, that obstructs a decomposable product system from being a CCR flow (even if we allow for twists by multipliers).  
The class of isometric representations for which we compute the space of $2$-cocycles and the set of admissible $2$-cocycles is next described.

 We assume that $d \geq 2$.  Let $P$ be a pointed, spanning closed convex cone in $\mathbb{R}^d$.  Let $H$ be a closed subgroup of $\mathbb{R}^d$. 
Note that $\mathbb{R}^d$ acts on the homogeneous space $\mathbb{R}^d/H$ by translations. Let $A \subset \mathbb{R}^d/H$ be a non-empty, proper, Borel subset such that $A+P \subset A$. We call such subsets $P$-spaces. 
Consider the Hilbert space $L^2(A)$. For $a \in P$, let $V_a$ be the isometry on $L^2(A)$ defined by the equation 
\begin{equation}
 \label{isometries intro}
V_{a}(f)(x):=\begin{cases}
 f(x-a)  & \mbox{ if
} x-a \in A,\cr
   &\cr
    0 &  \mbox{ if } x-a \notin A.
         \end{cases}
\end{equation}
Then, $V^A:=V=\{V_a\}_{a \in P}$ is a pure isometric representation of $P$. We call $V^{A}$ the shift semigroup associated to the $P$-space $A$. 

For the shift semigroup $V^{A}$, we prove the following. 
\begin{enumerate}
\item[(1)] Suppose $L(H):=span(H)$ has co-dimension one in $\mathbb{R}^d$. Then, 
\[
H^2(P,V^A)=L(H)\oplus L(H).\]
Here, we consider $L(H)\oplus L(H)$ as the complexification of the real vector space $L(H)$. Moreover, 
\[
H^2_a(P,V^A)=\{(\lambda,\mu) \in L(H) \oplus L(H): \textrm{$\lambda$ and $\mu$ are linearly dependent}\}.\] 
In this case, $\mathcal{D}_{V^A}(P)$ has a nice parametrisation. Let $\sim$ be the equivalence relation on $L(H)$ defined by $\lambda \sim \mu$ if $\lambda=\pm \mu$.
Then,
\[
\mathcal{D}_{V^A}(P)=\frac{L(H)}{\sim}.
\]
%Recall that for an isometric representation $V$, $\mathcal{D}_V(P)$ denotes the collection (up to projective isomorphism) of product systems over $P$ whose associated isometric representation is $V$.

\item[(2)] Suppose $L(H):=span(H)$ is of co-dimension greater than $1$. In this case, every $2$-cocycle for $V^{A}$ is admissible and \[
H^{2}(P,V^A)=H^{2}_{a}(P,V^A)=H^{1}(G,L^{2}(G))\] for a locally compact, abelian group $G$ which is not compact. Here, $H^1(G,L^2(G))$ is the non-reduced $L^2$-cohomology group of $G$. Moreover, the group $G$ depends  only on $H$ and does not depend on the $P$-space $A$. 
\end{enumerate}

The organization of this paper is next described. After this introductory section, in Section 2, we discuss the notion of algebraic $E_0$-semigroups and algebraic product systems which
we encounter at one crucial point in the paper. In Section 3, we discuss the space of $2$-cocycles. We show that there is a degree reduction in a certain sense. In particular, 
we prove that computing the second cohomology group amounts to computing a certain first cohomology group. In Section 4, we prove Thm. \ref{mainintro} and Thm. \ref{mainintro1}. 
In the last section, we compute the cohomology groups when the isometric representation is the shift semigroup associated to a transitive action of the semigroup $P$. 

Although this work could, in  part,  be rightly considered an imitation of Arveson's work (\cite{Arv_Path}), we strongly believe that it is essential, it is of intrinsic interest, and it is worth the effort to record, for future reference, the exact cohomological obstructions
that appear in the higher dimensional case and to compute them in concrete situations. Also, there is a pedantic subtlety regarding the definition of the decomposabiliy of a product system (see Remark \ref{pedantic remark}), caused by the lack of total order,
whose resolution demands an explicit description of the structure of a generic decomposable product system, and we do not know of any elementary
argument that resolves this subtlety. These are a couple of reasons for  spelling out explicitly the cohomological issues.

\section{Preliminaries}
 Let $P$ be closed convex cone in $\R^{d}$ which we assume is spanning, i.e $P-P=\bbr^d$ and pointed, i.e. $P \cap -P=\{0\}$. Denote the interior of $P$ by $\Omega$. The cone $P$ will remain fixed throughout this paper. For $a,b\in P$, we write $a\leq b$ if $b-a\in P$ and $a<b$ if $b-a\in \Omega$. For an interior point $a \in \Omega$, the sequence $\{na\}_{n \in \bbn}$ is cofinal. 
 Since we will need the notion of algebraic product systems later in the paper, 
 we make formal definitions in this section. 

Let $\mathcal{H}$ be an infinite dimensional, separable Hilbert space. An \emph{algebraic $E_{0}$-semigroup} over $P$ on $B(\mathcal{H})$ is a semigroup $\{\alpha_{a}\}_{a\in P}$ of normal, unital *-endomorphisms of $B(\mathcal{H})$. Suppose $\alpha=\{\alpha_{a}\}_{a\in P}$ and $\beta=\{\beta_{a}\}_{a\in P}$ are algebraic $E_{0}$-semigroups  on $B(\mathcal{H})$. We say that $\alpha$ is \emph{cocycle conjugate} to $\beta$ if there exists a family $\{U_{a}\}_{a \in P}$ of unitaries on  $\mathcal{H}$ such that
\begin{enumerate}[(1)]
\item  for $a,b\in P$, $U_{a}\alpha_{a}(U_{b})=U_{a+b}$, and
\item for $a \in P$, $\beta_{a}(.)=U_{a}\alpha_{a}(.)U_{a}^{*}$.
\end{enumerate}

An \emph{algebraic product system} over $P$ is a field $E:=\{E(a)\}_{a\in P}$ of separable Hilbert spaces along with an associative multiplication defined  on the disjoint union $\displaystyle{\coprod_{a\in P}}E(a)$ such that
\begin{enumerate}[(1)]
\item  if $u\in E(a)$ and $v\in E(b)$, then $uv\in E(a+b)$, and
\item  the map $E(a)\otimes E(b) \ni u \otimes v \mapsto uv \in E(a+b)$  is  a unitary for every $a,b \in P$.
\end{enumerate}
The notion of isomorphism between two algebraic product systems is 'clear' and we do not make the formal definition here.

Suppose $\alpha=\{\alpha_{a}\}_{a\in P}$ is an algebraic $E_{0}$-semigroup  on $B(\mathcal{H})$. 
For $a\in P$, define
\begin{equation*}
E^{\alpha}(a):=\{T \in B(\mathcal{H}): \textrm{$\alpha_{a}(A)T=TA$ for every $A\in B(\mathcal{H})$}\}.
\end{equation*}
Then, $E^{\alpha}(a)$ is a separable Hilbert space with the inner product defined 
by \[\langle S|T\rangle=T^{*}S.\] Also, $E^{\alpha}:=\{E^{\alpha}(a)\}_{a\in P}$ is an algebraic product system with the product given by the usual multiplication.  The algebraic product system $E^{\alpha}$ is called the algebraic product system associated to $\alpha$. Just like in the measurable situation,  two algebraic $E_{0}$-semigroups  are cocycle conjugate if and only if their associated product systems are isomorphic.

Let $E=\{E(a)\}_{a \in P}$ be an algebraic product system over $P$. A section $e:P\to \displaystyle{\coprod_{a\in P}}E(a)$ of non-zero vectors is said to be \emph{left coherent} if for $a,b\in P$ with $a\leq b$, there exists $e(a,b)\in E(b-a)$ such that $e_{a}e(a,b)=e_{b}$. If $e$ is a left coherent section, the collection $\{e(a,b):a \leq b\}$ is called the set of propagators of $e$. We also have  
the equation
\begin{equation}
\label{equation for propogators}
e(a,b)e(b,c)=e(a,c)
\end{equation}
whenever $a \leq b \leq c$.
\begin{prop}
\label{essential rep}
Suppose $E$ is an algebraic product system with a  a left coherent section. Then, $E$ is isomorphic to the product system of an algebraic $E_0$-semigroup.
\end{prop}
\begin{proof}
It suffices to construct an essential representation of $E$ on a separable Hilbert space. The proof is an application of the standard inductive limit construction, due to Arveson, that yields a quick proof of the fact that a $1$-parameter product system with a unit has an essential representation.

Suppose $e=\{e_{a}\}_{a\in P}$ is a left coherent section of $E$. We may assume that $||e_{a}||=1$ for each $a\in P$. %We construct a separable Hilbert space using this left coherent section, by an inductive limit technique.
 Suppose $a,b\in P$ and $a\leq b$. Consider the isometry $V_{b,a}:E(a)\to E(b)$ defined by \[V_{b,a}(u)=ue(a,b).\] If $a,b,c\in P$ are such that $a\leq b\leq c$, then $V_{c,a}=V_{c,b}V_{b,a}$. Denote the inductive limit of this directed system  of Hilbert spaces by $\mathcal{H}$. For $a\in P$, let $i_{a}:E(a)\to\mathcal{H}$ be the inclusion map.
 Fix $a_0\in\Omega$. The fact that $\{na_0\}_{n \in \bbn}$ is cofinal implies  that $\bigcup_{n\in\mathbb{N}}i_{na_0}(E(na_0))$ is dense in $\mathcal{H}$. Hence, $\mathcal{H}$ is separable. 

 Let $a\in P$ and $u\in E(a) $. Define $\theta_{a}(u) \in  B(\mathcal{H})$ by \[\theta_a(u)(i_{b}(v))=i_{a+b}(uv)\] for $b\in P$ and $v\in E(b)$. Then, $\theta=\{\theta_a\}_{a \in P}$ defines a representation  of $E$ on $\mathcal{H}$. 

 We now prove that the representation $\theta$ is essential. Let $a\in P$ be given. It suffices to show that $i_{b}(u)\in [\theta_a(E(a))\mathcal{H}]$ for any $b\geq a $ and for $u \in E(b)$. Here, $[\theta_a(E(a))\mathcal{H}]$ denotes the closed linear span of $\{\theta_a(x)\xi: x \in E(a), \xi \in \mathcal{H}\}$.
  Since vectors of the form $u=vw$, where $v\in E(a)$ and $w\in E(b-a)$, form a total set in $E(b)$, it suffices to prove the statement for such vectors. But if $u=vw$ with $v\in E(a)$ and $w\in E(b-a)$, then  $i_{b}(u)=\theta_{a}(v)i_{b-a}(w)\in \theta_a(E(a))\mathcal{H}$.  The proof is complete.
\end{proof}

\begin{prop}
\label{over omega implies over P}Suppose $\alpha=\{\alpha_{a}\}_{a\in P}$, and $\beta=\{\beta_{a}\}_{a\in P}$ are two algebraic $E_{0}$-semigroups  on $B(\mathcal{H})$, where $\mathcal{H}$ is a  separable Hilbert space. Then, if $\alpha$ and $\beta$ are cocycle conjugate over $\Omega$, then they are cocycle conjugate over $P$.
\end{prop}
\begin{proof} 
The hypothesis means that there exists a family of unitaries  $U=\{U_{a}\}_{a\in \Omega}$  in $B(\mathcal{H})$ such that $U_{a}\alpha_{a}(U_{b})=U_{a+b}$ for $a,b\in\Omega$, and $\beta_{a}(.)=U_{a}\alpha_a(.)U_{a}^{*}$ for $a\in\Omega$. 

Fix $a_{0}\in \Omega$. For $a\in P$, define
\[\widetilde{U}_{a}=U_{a+a_{0}}\alpha_{a}(U_{a_{0}})^{*}.\]

 Suppose $a_{0}$, $b_{0}\in\Omega$ and $a_{0}<b_{0}$. For $a\in P$, calculate as follows to observe that 
 \begin{align*}
 U_{a+b_{0}}^{*}U_{a+a_{0}}=& \alpha_{a+a_{0}}(U_{b_{0}-a_{0}})^{*}U_{a+a_{0}}^{*}U_{a+a_{0}}\\=&\alpha_{a+a_{0}}(U_{b_{0}-a_{0}}^{*})\\
 =&\alpha_{a}(\alpha_{a_{0}}(U_{b_{0}-a_{0}}^{*})U_{a_{0}}^{*}U_{a_{0}})\\
 =&\alpha_{a}(U_{b_{0}}^{*}U_{a_{0}})=\alpha_{a}(U_{b_{0}})^{*}\alpha_{a}(U_{a_{0}}).
 \end{align*}
 Hence, $U_{a+b_0}\alpha_a(U_{b_0})^{*}=U_{a+a_0}\alpha_a(U_{a_0})^{*}$.
Now for any $a_{0},b_{0}\in\Omega$, let $c_0\in\Omega$ be such that $a_0<c_0$ and $b_0<c_0$. Then, by the above calculation,
\[
U_{a+a_{0}}\alpha_{a}(U_{a_{0}})^{*}=U_{a+c_0}\alpha_{a}(U_{c_0})^{*}
               =U_{a+b_{0}}\alpha_{a}(U_{b_{0}})^{*}.
\]
Thus, the definition of $\widetilde{U}$ is independent of the choice of the interior point $a_{0}$.

It is clear that $\widetilde{U}_a=U_a$ if $a \in \Omega$. Hence, for $a \in P$ and $a_0 \in \Omega$, $\widetilde{U}_{a+a_{0}}=\widetilde{U}_{a}\alpha_{a}(\widetilde{U}_{a_{0}})$.
Also, for $a \in P$ and $a_0 \in \Omega$, 
\begin{align*}\widetilde{U}_{a_{0}}\alpha_{a_{0}}(\widetilde{U}_{a})=&\widetilde{U}_{a_{0}}\alpha_{a_{0}}(\widetilde{U}_{a+a_{0}}\alpha_{a}(\widetilde{U}_{a_{0}}^{*}))\\
                   =&\widetilde{U}_{(a+a_{0})+a_{0}}\alpha_{a+a_{0}}(\widetilde{U}_{a_0}^{*})\\=&\widetilde{U}_{a+a_{0}}.
\end{align*}
Let $a,b \in P$ be given. Observe that
\begin{align*}
\widetilde{U}_{a}\alpha_{a}(\widetilde{U}_{b})=& \widetilde{U}_{a+a_{0}}\alpha_{a}(\widetilde{U}_{a_{0}})^{*}\alpha_{a}(\widetilde{U}_{b+a_{0}})\alpha_{a+b}(\widetilde{U}_{a_{0}})^{*}\\
=& \widetilde{U}_{a+a_{0}}\alpha_{a+a_{0}}(\widetilde{U}_{b})\alpha_{a+b}(\widetilde{U}_{a_{0}})^{*}\\
=&\widetilde{U}_{a+b+a_0}\alpha_{a+b}(\widetilde{U}_{a_0})^{*}\\
=& \widetilde{U}_{a+b}.
\end{align*}
 
Define an algebraic $E_{0}$-semigroup $\widetilde{\beta}$ over $P$  on  $B(\mathcal{H})$ by $\widetilde{\beta}_{a}(.)=\widetilde{U}_{a}\alpha_{a}(.)\widetilde{U}_{a}^{*}$ for $a\in P$. Clearly, $\widetilde{\beta}_{a}=\beta_{a}$ for any $a\in \Omega$. In particular,  
$\widetilde{\beta}_{a+a_{0}}=\beta_{a+a_{0}}$ for $a\in P$, and
$\widetilde{\beta}_{a_{0}}=\beta_{a_{0}}$.
  Since 
$\beta_{a_{0}}\circ \widetilde{\beta}_{a}=\beta_{a_{0}} \circ \beta_{a}$ for $a\in P$ and $\beta_{a_0}$ is injective, $\widetilde{\beta}_a=\beta_a$ for every $a \in P$. This completes the proof. 
\end{proof}

The following corollary is immediate from Prop. \ref{essential rep} and Prop. \ref{over omega implies over P}.
\begin{corollary}
\label{over omega implies over P1}Let $E=\{E(a)\}_{a\in P}$ and $F=\{F(a)\}_{a\in P}$ be two algebraic product systems, each possessing a left coherent section. If $E$ and $F$ are isomorphic over $\Omega$, then they are isomorphic over $P$.
\end{corollary}

Let $E:=\{E(a)\}_{a \in P}$ be an algebraic product system. 
Let $b\in P$. A non-zero vector $u\in E(b)$ is said to be decomposable if for any $a\in P$ such that $a\leq b$, there exists $u_{a}\in E(a)$ and $u(a,b)\in E(b-a)$ such that $u=u_{a} u(a,b)$.  For $a\in P$, let  $D(a)$ be the set of all decomposable vectors in $E(a)$. The algebraic product system $E$ is said to be \emph{decomposable} if the following conditions hold.
\begin{enumerate}[(1)]
\item For each $a \in P$, $D(a)$ is total in $E(a)$, and
\item for $a,b\in P$, $D(a)D(b)= D(a+b)$.
\end{enumerate}

We end this section by recalling the definition of a `genuine' $E_0$-semigroup 
and a `genuine' product system. Let $\alpha:=\{\alpha_a\}_{a \in P}$ be an algebraic $E_0$-semigroup on $B(\mathcal{H})$. We say that $\alpha$ is a `genuine'
$E_0$-semigroup, or simply an \emph{$E_0$-semigroup} if for $A \in B(\mathcal{H})$ and $\xi,\eta \in \mathcal{H}$, the map 
\[
P \ni  a \to \langle\alpha_a(A)\xi|\eta \rangle \in \mathbb{C}
\]
is measurable. We will always use the adjective `algebraic' to refer to $E_0$-semigroups that are not `genuine'.

Let $E:=\{E(a)\}_{a \in P}$ be an algebraic product system. Suppose, in addition, $E$ is a measurable field of Hilbert spaces with $\mathcal{S}$ being the space of measurable sections. We say that $E$ together with $\mathcal{S}$ is a `genuine' product system, or simply a \emph{product system} if for $r,s,t \in \mathcal{S}$, the map 
\[
P \times P \ni (a,b) \to \langle r(a)s(b)|t(a+b)\rangle \in \mathbb{C}\]
is measurable. Again, we always use the adjective `algebraic' to refer to product systems that are not genuine. 

\section{The space of cocycles}
In this section, we will discuss about the space of cocycles associated to a strongly continuous semigroup of isometries.  The results from this section will be put to use in subsequent sections.
Let $V=\{V_{a}\}_{a\in P}$ be a pure isometric representation of $P$ on a Hilbert space $\mathcal{H}$. The isometric representation $V$ is fixed for the rest of this section. Let us recall the definition of $2$-cocycles mentioned in the introduction.

\begin{definition}A map $\Gamma:P\times P\to \mathcal{H}$ is called a \emph{2-cocycle} for $V$ (or simply a $2$-cocycle) if
\begin{enumerate}[(1)]
\item for $a,b\in P$, $\Gamma(a,b)\in Ker(V_{b}^{*})$,
\item the map $\Gamma$ satisfies the following  identity
\begin{equation*}\Gamma(a,b)+V_{b}\Gamma(a+b,c)-\Gamma(a,b+c)-V_{b}\Gamma(b,c)=0
\end{equation*} for every $a,b,c\in P$, and
\item for each $\xi\in\mathcal{H}$, the map $P\times P\ni (a,b)\to\langle\Gamma(a,b)|\xi\rangle \in \mathbb{C}$ is Borel.
\end{enumerate}
\end{definition}
 Let $Z^{2}(P,\mathcal{H})$ be the space of all 2-cocycles. A 2-cocycle $\Gamma$ is said to be a \emph{coboundary} if there exists a Borel map $P\ni a\to \xi_{a}\in\mathcal{H}$ such that $\xi_{a}\in Ker(V_{a}^{*})$ for each $a\in P$, and
\begin{equation*}\label{gamma}
V_{a}\Gamma(a,b)= \xi_{a+b}-\xi_{a}-V_{a}\xi_{b}
\end{equation*} for every $a,b\in P$. The space of coboundaries, denoted $B^{2}(P,\mathcal{H})$, is a subspace of $Z^{2}(P,\mathcal{H})$. Define the cohomology group $H^{2}(P,\mathcal{H})$ by 
\begin{equation*}
H^{2}(P,\mathcal{H}):=\frac{Z^{2}(P,\mathcal{H})}{B^{2}(P,\mathcal{H})}.
\end{equation*}
To denote the dependence of $H^{2}(P,\mathcal{H})$ on $V$, we denote $H^{2}(P,\mathcal{H})$ by $H^{2}(P,\mathcal{H},V)$, or by $H^{2}(P,V)$.

\begin{remark}
\label{continuity of coherence}
Let $\xi:P \to \mathcal{H}$ be a map. The map $\xi$ is said to be coherent if for $a \leq b$, $\xi_a=E_{a}^{\perp}\xi_b$. Here, $E_a=V_aV_{a}^{*}$. Note that if $\xi:P \to \mathcal{H}$ is coherent, then $\xi$ is continuous. 
To see this, fix $b \in \Omega$. Set \[F_n:=\{a \in P: a \leq nb\}.\]  For $a \in F_n$, the equality $\xi_a=E_{a}^{\perp}\xi_{nb}$ implies that the map 
\[
F_n \ni a \to \xi_a \in \mathcal{H}
\]
is  continuous. Since $\{F_n\}_{n \geq 1}$ is an open cover of  $P$, the map 
$\xi$ is continuous. It is also clear that $\displaystyle \lim_{a \to 0}\xi_a=0$.

\end{remark}

Next, we introduce the coordinate free  analogue  of the space $L^{2}_{loc}((0,\infty),\mathcal{C})$ considered in Section 5.3 of \cite{arveson}.
For each $a\in P$, let $E_{a}$ denote the range projection of $V_{a}$, i.e. $E_{a}=V_{a}V_{a}^{*}$. 
Let $\displaystyle \mathcal{D}_{V}:=\bigcup_{a\in P}Ker(V_{a}^{*})$.  Since $V$ is pure, $\mathcal{D}_V$ is a dense subspace of $\mathcal{H}$. Set
\begin{equation*}
\mathcal{L}_{V}:=\{\varphi:D_{V}\to\mathbb{C}: \textrm{ $\varphi$ is linear and $\varphi|_{Ker(V_a^*)}$ is   bounded  for each $a\in P$}\}.
\end{equation*}

Let $a\in P$. For $\varphi \in \mathcal{L}_V$, define $||\varphi||_{a}$ by
\begin{equation*}
||\varphi||_{a}:=\big|\big|\varphi|_{Ker(V_a^*)}\big|\big|.
\end{equation*} 
Then, $||.||_{a}$ is a seminorm on $\mathcal{L}_{V}$. The collection of seminorms $\{||.||_{a}: a \in P\}$ induce a locally convex topology $\tau_{V}$ on $\mathcal{L}_{V}$. If $a\leq b$, then $Ker( V_{a}^{*})\subset Ker(V_{b}^{*})$. Hence, if $a\leq b$, then  $||\varphi||_{a}\leq ||\varphi||_{b}$ for every $\varphi \in \mathcal{L}_V$.  For $b\in \Omega$, the 
 topology defined by the collection of  seminorms $\{||.||_{nb}: n \in\bbn \}$ coincides with $\tau_{V}$. It is not difficult to prove that $(\mathcal{L}_{V},\tau_{V})$ is a Frechet space.

Define a   representation $T$ of $P$ on $\mathcal{L}_{V}$ as follows. For $a \in P$, let $T_a$ be the linear operator on $\mathcal{L}_V$ defined by the equation
\begin{equation*}
T_{a}\varphi(\xi):=\varphi(V_{a}\xi)
\end{equation*} for  $\varphi\in\mathcal{L}_{V}$ and $\xi\in D_{V}$.

\begin{definition}
\label{definition of 1-cocycle}
Let $\phi:P\to \mathcal{L}_{V}$ be a  map. Then, $\phi$ is called a \emph{1-cocycle} if
\begin{enumerate}[(1)]
\item $\phi_{a+b}=\phi_{a}+T_{a}\phi_{b}$ for each $a,b\in P$, and
\item the map $P\ni a\to \phi_{a}(\xi)\in \bbc$ is measurable for every $\xi\in D_{V}$.
\end{enumerate}
\end{definition}
 The above definition is the coordinate free version of Defn. 5.3.1 of \cite{arveson}.

\begin{remark}
\label{continuity of cocyclesblah}
Let $\phi:P \to \mathcal{L}_V$ be a $1$-cocycle. Then, $\phi$ is continuous. The proof is  the same as that of the $1$-parameter case given in Lemma 5.3.5 of \cite{arveson}. 
\end{remark}

Let $Z^{1}(P,\mathcal{L}_{V})$ denote the vector space of all 1-cocycles. 
A 1-cocycle $\phi$ is called a coboundary if there exists $\varphi\in \mathcal{L}_{V}$ such that for $a \in P$,
\begin{equation*}
\phi_{a}=T_{a}\varphi-\varphi.
\end{equation*}
Let $B^{1}(P,\mathcal{L}_{V})$ be the set of coboundaries. Note that $Z^{1}(P,\mathcal{L}_{V})$ is a  vector space, and $B^{1}(P,\mathcal{L}_{V})$ is a subspace. Denote the quotient $\frac{Z^{1}(P,\mathcal{L}_{V})}{B^{1}(P,\mathcal{L}_{V})}$ by $H^{1}(P,\mathcal{L}_{V})$. 

Suppose $\phi\in Z^{1}(P,\mathcal{L}_{V})$. For $a,b\in P$, let $\Gamma^{\phi}(a,b)\in Ker(V_{b}^{*})$ be such that
\begin{equation*}
 \langle\xi|\Gamma^{\phi}(a,b)\rangle:=\phi_a(\xi)
\end{equation*} for $\xi\in Ker(V_{b}^{*})$.
We denote the map $P \times P \ni (a,b)\to \Gamma^{\phi}(a,b) \in \mathcal{H}$ by $\Gamma^{\phi}$. It is clear that $\Gamma^{\phi}$
is measurable.

The following proposition is a more formal, coordinate independent, cohomological expression of 
Lemma 5.5.9 of \cite{arveson}.

\begin{prop} \label{equiv}
Keep the foregoing notation.
Suppose $\phi\in Z^{1}(P,\mathcal{L}_{V})$. Then, $\Gamma^{\phi}$ is a 2-cocycle.
Conversely, suppose $\Gamma:P\times P\to\mathcal{H}$ is a 2-cocycle, then there exists a unique 1-cocycle $\phi$ such that $\Gamma=\Gamma^{\phi}$. 
Also, $\phi$ is a coboundary if and only if $\Gamma^{\phi}$ is a coboundary. Consequently,  the map $\theta:H^1(P,\mathcal{L}_V) \to H^{2}(P,\mathcal{H})$ defined by $\theta([\phi])=[\Gamma^{\phi}]$ is an isomorphism. 
\end{prop}
\begin{proof}
Let $\phi\in Z^{1}(P,\mathcal{L}_{V})$. 
 Then, \[ \phi_{a+b}(\xi)=\phi_{a}(\xi)+\phi_{b}(V_{a}\xi)\] for any $a,b\in P$ and $\xi \in D_{V}$.
 Denote $\Gamma^{\phi}$ simply by $\Gamma$.
Observe that
\begin{align*}
\langle\xi|\Gamma(a,b)\rangle+\langle \xi|V_b\Gamma(a+b,c)\rangle&= \phi_{a}(E_{b}^{\perp}\xi)+\phi_{a+b}(E_{c}^{\perp}V_{b}^{*}\xi)\\
 &=\phi_{a}(E_{b}^{\perp}\xi)+\phi_{b}(E_{c}^{\perp}V_{b}^{*}\xi)+T_b\phi_{a}(E_{c}^{\perp}V_{b}^{*}\xi) \\
 %&~~~(\textrm{by the cocycle identity})\\
 &=\phi_{a}(E_{b}^{\perp}\xi)+\phi_{b}(E_{c}^{\perp}V_{b}^{*}\xi)+\phi_{a}(V_{b}E_{c}^{\perp}V_{b}^{*}\xi)\\
            &= \phi_{a}(E_{b+c}^{\perp}\xi)+\phi_{b}(E_{c}^{\perp }V_{b}^{*}\xi)\\
            &=\langle\xi|\Gamma(a,b+c)\rangle+\langle \xi|V_{b}\Gamma(b,c)\rangle
\end{align*} for any $\xi\in D_V$.
Thus, $\Gamma$ satisfies the cocycle identity.

Let $\Gamma$ be a $2$-cocycle. Fix $a\in P$.  Define $\phi_{a}:D_{V}\to \mathbb{C}$ by
\begin{equation*}
\phi_{a}(\xi):=\langle\xi|\Gamma(a,b)\rangle
\end{equation*} 
whenever $\xi \in Ker(V_b^*)$.
We claim that $\phi$ is a well-defined 1-cocycle and $\Gamma=\Gamma^{\phi}$.
Suppose $\xi\in ker(V_{b}^{*})$ and $\xi\in Ker(V_{c}^{*})$ for some $b, c\in P$.
By the cocycle identity, we have
$\Gamma(a,b)=E_{b}^{\perp}\Gamma(a,b+c)$ and $\Gamma(a,c)=E_{c}^{\perp}\Gamma(a,b+c)$.
Then, \begin{align*}
\langle \xi|\Gamma(a,b)\rangle=&\langle \xi|E_{b}^{\perp}\Gamma(a,b+c)\rangle\\
=&\langle E_{b}^{\perp}\xi|\Gamma(a,b+c) \rangle\\
=&\langle \xi|\Gamma(a,b+c)\rangle\\
=&\langle E_{c}^{\perp}\xi|\Gamma(a,b+c)\rangle\\
=&\langle \xi|E_{c}^{\perp}\Gamma(a,b+c)\rangle\\
=&\langle \xi|\Gamma(a,c)\rangle.
    \end{align*}
    This proves that $\phi_{a}$ is well-defined for each $a\in P$. By definition, it is clear that $\phi_{a}$ is bounded on $Ker(V_{b}^{*})$ for each $b\in P$. Therefore, $\phi_a \in \mathcal{L}_V$ for $a \in P$. Let $\phi:P\to\mathcal{L}_{V}$ be the map $P\ni a\to \phi_{a} \in \mathcal{L}_V$.
As the map \[P \times P \ni (a,b) \to \langle \xi|\Gamma(a,b)\rangle \in \mathbb{C}\] is measurable for every  $\xi\in \mathcal{H}$, it follows that  $P\ni a\to \phi_{a}(\xi)\in\mathbb{C}$ is Borel for each $\xi\in \mathcal{D}_V$. 
 
We now show that $\phi$ satisfies the cocycle identity. Let $c\in P$, and  let $\xi\in Ker(V_{c}^{*})$. Then, for $a,b \in P$,
\begin{align*}
\phi_{a+b}(\xi)=&\langle\xi|\Gamma(a+b,c)\rangle \\
=&\langle\xi|V_{b}^{*}\Gamma(a,b+c)\rangle+\langle\xi|\Gamma(b,c)\rangle~~(\textrm{by the  cocycle identity})\\
=& T_{b}\phi_{a}(\xi) +\phi_{b}(\xi).
\end{align*}
Hence, $\phi$ is a $1$-cocycle. Clearly, $\Gamma=\Gamma^{\phi}$.

It is clear that the map $Z^{1}(P,\mathcal{L}_V) \ni \phi \to \Gamma^{\phi} \in Z^{2}(P,\mathcal{H})$ is $1$-$1$, and we have just shown that it is onto. 
 Suppose $\phi \in B^{1}(P,\mathcal{L}_{V})$. Then, there exists $\varphi\in\mathcal{L}_{V}$ such that
 \begin{equation*}
 \phi_{a}=T_{a}\varphi-\varphi.
 \end{equation*} for each $a\in P$.
 Then, for $\xi\in D_{V}$,
 \begin{equation}\label{welldefined}
\langle\xi|V_{a}\Gamma^{\phi}(a,b)\rangle=(T_{a}\varphi-\varphi)(E_{b}^{\perp}V_{a}^{*}\xi)
=\varphi(E_{a+b}^{\perp}\xi)-\varphi(E_{a}^{\perp}\xi)-\varphi(E_{b}^{\perp}V_{a}^{*}\xi)
\end{equation} for $a,b\in P$. 
 For $a\in P$, let $\xi_{a}\in Ker(V_{a}^{*})$ be such that 
 \begin{equation*}
 \varphi(\xi)=\langle\xi|\xi_{a}\rangle
 \end{equation*} for each $\xi \in Ker (V_{a}^{*})$. Then, Eq. \ref{welldefined} implies that
 \begin{align*}
 \langle\xi|V_{a}\Gamma^{\phi}(a,b)\rangle=&\langle\xi|\xi_{a+b}-\xi_{a}-V_{a}\xi_{b}\rangle
\end{align*} for each $\xi\in \mathcal{D}_V$ and for $a,b\in P$. Since $\mathcal{D}_V$ is dense in $\mathcal{H}$, we have \begin{equation}
\label{cobound}V_{a}\Gamma^{\phi}(a,b)=\xi_{a+b}-\xi_{a}-V_{a}\xi_{b}\end{equation} for each $a,b\in P$. 

Note that Eq. \ref{cobound} implies that for $a \leq b$, $\xi_a=E_{a}^{\perp}\xi_b$, i.e. $\xi$ is coherent. By Remark \ref{continuity of coherence}, it follows that $\xi$ is Borel.   Now Eq. \ref{cobound} implies that  $\Gamma^{\phi}$ is a coboundary. This shows that the map $\theta$ is well-defined.

 Suppose $\phi\in H^{1}(P,\mathcal{L}_{V})$ is such that $\Gamma^{\phi}$ is a coboundary. Then, there exists a Borel map $P\ni a\to \xi_{a}\in\mathcal{H}$ such that $\xi_a \in Ker(V_a^*)$ for $a \in P$ and for $a,b \in P$ 
 \begin{equation}\label{injectivity}
V_{a}\Gamma^{\phi}(a,b)=\xi_{a+b}-\xi_{a}-V_{a}\xi_{b}.
 \end{equation}
Define $\varphi:D_{V}\to \mathbb{C}$ by
\begin{equation*}
\varphi(\xi):=\begin{cases}\langle\xi|\xi_{a}\rangle & \textrm{if $\xi \in Ker( V_{a}^{*})$}.
\end{cases}
\end{equation*}
Eq. \ref{injectivity} implies that $E_{b}^{\perp}\xi_{b+c}=\xi_b$. Using this, it is not difficult to check that  that $\varphi$ is well defined. By the definition of $\varphi$, it is clear that $\varphi$ is bounded on $Ker(V_{b}^{*})$ for each $b\in P$.

 Thanks to  Eq. \ref{injectivity}, for $\xi\in Ker (V_{b}^{*})$ and $a\in P$,
 \begin{align*}
 \phi_{a}(\xi)&=\langle \xi|\Gamma^{\phi}(a,b)\rangle \\
 &=\langle \xi|V_{a}^{*}\xi_{a+b}-\xi_b \rangle\\
 &=\langle V_{a}\xi|\xi_{a+b}\rangle -\langle \xi|\xi_b \rangle\\
 &=\varphi(V_{a}\xi)-\varphi(\xi)\\
 &=(T_{a}\varphi-\varphi)(\xi).
 \end{align*}
 Hence, $\phi$ is a coboundary  whenever $\Gamma^{\phi}$ is a coboundary. In other words, $\theta$ is injective.
  The proof is complete.
 \end{proof}
 \begin{lemma}  Let $\Gamma$ be a $2$-cocycle. Then, $\Gamma$ is norm continuous.
 \end{lemma}
 \begin{proof}
Let $\phi \in Z^1(P,\mathcal{L}_V)$ be such that $\Gamma=\Gamma^{\phi}$. It follows from  Remark \ref{continuity of cocyclesblah} that the map $P \ni a \to \phi_a \in \mathcal{L}_V$ is continuous. By the definition of $\Gamma^{\phi}$
and the continuity of $\phi$, it follows that for a fixed $b \in P$, the map 
\[
P \ni a \to \Gamma(a,b) \in Ker(V_b^*) \subset \mathcal{H}
\]
is norm continuous. 

 For $a\in P$,   observe from the cocycle identity that the set $\{\Gamma(a,b)\}_{b\in P}$ is coherent, i.e. if $b_{1}, b_{2}\in P$ are such that $b_{1}\leq b_{2}$, then 
 \begin{equation*}
 \Gamma(a,b_{1})=E_{b_{1}}^{\perp}\Gamma(a,b_{2}).
 \end{equation*} 
Let $a_0 \in \Omega$ be given. Let $F_n:=\{(a,b) \in P \times P: a,b <na_0\}$. Then, $\{F_n\}_{n \geq 1}$ is an open cover of $P\times P$. Observe that for $(a,b) \in F_n$, $\Gamma(a,b)=E_{b}^{\perp}\Gamma(a,na_0)$. The proof follows from the strong continuity of the map $P \ni b \to E_{b}^{\perp} \in B(\mathcal{H})$ and the continuity of the map $P \ni a \to \Gamma(a,na_0) \in \mathcal{H}$.
 \end{proof}
 Note that we can define the spaces $Z^{1}(\Omega, \mathcal{L}_V)$, $B^1(\Omega,\mathcal{L}_V)$ and $H^1(\Omega,\mathcal{L}_V)$ in an analogous 
 manner. Similarly, we can define the spaces $Z^{2}(\Omega,\mathcal{H})$, $B^{2}(\Omega,\mathcal{H})$ and $H^{2}(\Omega, \mathcal{H})$ and we have the isomorphism 
 $H^{1}(\Omega,\mathcal{L}_V)\cong H^{2}(\Omega,\mathcal{H})$.
 
Let $r:Z^{1}(P,\mathcal{L}_V) \to Z^{1}(\Omega, \mathcal{L}_V)$ be the restriction map. 

\begin{prop}\label{extension}
The restriction map $r:Z^{1}(P,\mathcal{L}_V) \to Z^{1}(\Omega,\mathcal{L}_V)$ is an isomorphism. Also, it
descends to establish an isomorphism between $H^1(P,\mathcal{L}_V)$ and $H^{1}(\Omega,\mathcal{L}_V)$.
\end{prop}
\begin{proof}
The continuity of $1$-cocycles and the density of $\Omega$ in $P$ imply that $r$ is $1$-$1$. Next we show that $r$ is surjective.  Suppose $\delta\in Z^{1}(\Omega,\mathcal{L}_{V})$.  Fix $a_{0}\in \Omega$. 
   Define $\phi:P\to\mathcal{L}_{V}$ by
 \begin{equation*}
 \phi_{a}=\delta_{a+a_0}-T_{a}\delta_{a_0}.
 \end{equation*} 
 It is clear that $\phi_a=\delta_a$ for $a \in \Omega$. Also, the map 
 $P \ni a \to \phi_a \in \mathcal{L}_V$ is continuous. 
 Since $\phi_{a+b}=\phi_a+T_a\phi_b$ for $a,b \in \Omega$ and $\Omega$ is dense in $P$, it follows that $\phi_{a+b}=\phi_a+T_{a}\phi_b$ for $a,b \in P$.
 Therefore, $\phi \in Z^1(P,\mathcal{L}_V)$ and $r(\phi)=\delta$.  Hence, $r:Z^1(P,\mathcal{L}_V) \to Z^1(\Omega, \mathcal{L}_V)$ is surjective. Consequently, $r:Z^1(P,\mathcal{L}_V) \to Z^1(\Omega,\mathcal{L}_V)$ is an isomorphism.

It is clear that $r$ descends to a linear map from $H^1(P,\mathcal{L}_V)$ to $H^1(\Omega,\mathcal{L}_V)$ which we again denote by $r$. We have already show that $r$ is surjective. We claim that $r$ is injective.  Suppose $r(\phi) \in B^1(\Omega,\mathcal{L}_V)$ for some $\phi \in Z^1(P,\mathcal{L}_V)$. Then, there exists  $\varphi\in \mathcal{L}_{V}$  such that
\begin{equation*}
\phi_{a}=T_{a}\varphi-\varphi
\end{equation*} for $a\in\Omega$.
Consider the coboundary $\psi \in Z^{1}(P,\mathcal{L}_{V})$ defined by
\begin{equation*}
\psi_{a}=T_{a}\varphi-\varphi
\end{equation*} for $a\in P$.
Then, $\phi=\psi$ on $\Omega$. Since $1$-cocycles are continuous and $\Omega$ is dense in $P$, it follows that $\phi_{a}=\psi_{a}$ for $a \in P$. This proves the injectivity of $r$.
  \end{proof}
The next corollary is an immediate consequence of the isomorphism between the space of $1$-cocycles and the space of $2$-cocycles.
\begin{corollary}
\label{extension third section}
The restriction map $Res:Z^{2}(P,\mathcal{H}) \to Z^{2}(\Omega,\mathcal{H})$ is an isomorphism. Moreover, $Res$
descends to establish an isomorphism between $H^2(P,\mathcal{H})$ and $H^2(\Omega,\mathcal{H})$.
\end{corollary}

\begin{definition}
Let $\Gamma \in Z^2(P,\mathcal{H})$. We say that $\Gamma$ is admissible if there exists a Borel map $\alpha:P\times P\to\mathbb{T}$ such that 
 \begin{equation}\label{admissibility}
 e^{iIm\langle \Gamma(a,b+c)|V_{b}\Gamma(b,c)\rangle}=\frac{\alpha(a,b)\alpha(a+b,c)}{\alpha(a,b+c)\alpha(b,c)}
 \end{equation} for each $a,b,c\in P$.
 A $1$-cocycle $\phi \in Z^1(P,\mathcal{L}_V)$ is said to be admissible if $\Gamma^\phi$ is admissible. 
\end{definition}

For an admissible 2-cocycle $\Gamma$, the map $\alpha:P\times P\to\mathbb{T}$ satisfying Eq. \ref{admissibility} is not unique. If $\alpha_{1}$ and $\alpha_{2}$ are two such maps, then there exists a multiplier $\omega$ such that $\alpha_{1}=\omega\alpha_{2}$.

We denote the set of admissible cocycles in $Z^2(P,\mathcal{H})$ by $Z_a^2(P,\mathcal{H})$, and the set of admissible $1$-cocycles in $Z^1(P,\mathcal{L}_V)$ will be denoted by $Z_a^1(P,\mathcal{L}_V)$. The image of $Z_{a}^{2}(P,\mathcal{H})$ in $H^{2}(P,\mathcal{H})$ and the image of $Z_{a}^{1}(P,\mathcal{L}_V)$ in $H^1(P,\mathcal{L}_V)$ will be denoted by $H_{a}^{2}(P,\mathcal{H})$ and $H^1_a(P,\mathcal{L}_V)$ respectively. The image of $Z^{2}_{a}(P,\mathcal{H})$ in $Z^{2}(\Omega,\mathcal{H})$ under the natural restriction map will be denoted $Z_{a}^{2}(\Omega,\mathcal{H})$. The notation $H^1_{a}(\Omega,\mathcal{L}_V)$ and $H^2_{a}(\Omega,\mathcal{L}_V)$ stand
for similar things.

Let $M_V:=\{V_x,V_{x}^{*}:x \in P\}^{'}$ and let $\mathcal{U}(M_V)$ be the unitary group of the von Neumann algebra $M_V$. The group $\mathcal{U}(M_V)$ acts naturally on the cohomology groups $H^1(P,\mathcal{L}_V)$, $H^{2}(P,\mathcal{H})$ and leaves
$H_{a}^{1}(P,\mathcal{L}_V)$ and $H_{a}^{2}(P,\mathcal{H})$ invariant.

\section{Structure of a decomposable product system}

We show in this section that a generic decomposable product system, up to a projective isomorphism, is determined
by an isometric representation and a certain $2$-cocycle. We also obtain necessary and sufficient conditions for two such product systems to be projectively isomorphic. 
 
Let $V=\{V_{a}\}_{a\in P}$ be a pure isometric representation  on a Hilbert space $\mathcal{H}$. Consider the product system $E^{V}$ associated to the corresponding CCR flow $\alpha^{V}$ acting on $\Gamma_{s}(\mathcal{H})$. 
 Recall that $E^{V}=\{E(a)\}_{a\in P}$, where $E^{V}(a)=E(a)=\Gamma_{s}(Ker(V_{a}^{*}))$ for each $a\in P$, and the multiplication rule on the exponential vectors is given by
\begin{equation*}
e(\xi)\odot e(\eta)=e(\xi+V_{a}\eta)
\end{equation*}for  $a,b\in P$, $\xi\in Ker(V_{a}^{*})$ and $\eta\in Ker(V_{b}^{*})$.
The symbol $\odot$ is reserved for the product rule of the CCR flow for the rest of this paper. By Prop. 3.2 of \cite{SUNDAR}, $E^{V}$ is decomposable, and for $a\in P$, the set of decomposable vectors in $E^{V}(a)$, denoted $D^V(a)$, is given by
\begin{equation}\label{ccrdecomposable}
D^{V}(a):= \{\lambda e(\xi):\textit{ $\xi\in Ker(V_{a}^{*})$, $\lambda\in\mathbb{C}\setminus \{0\}$}\}.      
\end{equation}

Suppose $\Gamma:P\times P \to \mathcal{H}$ is an admissible 2-cocycle. Let $\alpha:P\times P\to\mathbb{T}$ be a map that satisfies Eq. \ref{admissibility}.  Define a new product rule on $E^V$ as follows. The product rule on the exponential vectors is given by 
\begin{equation}\label{multiplication}
 e(\xi). e(\eta)=  \alpha(a,b)e(\xi)\odot W(\Gamma(a,b))e(\eta)
\end{equation} for $\xi \in Ker(V_a^*)$ and $\eta \in Ker(V_b^*)$.  The resulting product system will be denoted by $E^{(\alpha,\Gamma,V)}$ to indicate its dependence on the triple $(\alpha,\Gamma,V)$. The measurable structure on $E^{(\alpha,\Gamma,V)}$ is the same as that of $E^V$. 

The proof of the following proposition is elementary and a proof for a special case
can be found in \cite{Poisson} (Prop. 3.2 of \cite{Poisson}). We repeat the proof for 
completeness.  Keep the foregoing notation.
 \begin{prop}\label{->}
 The product system $E^{(\alpha,\Gamma,V)}$ is decomposable.
 \end{prop}
   \begin{proof}
  We will denote $E^{(\alpha,\Gamma,V)}$ by $E$ for simplicity.   It is routine to verify that the space $E$ with the above  multiplication rule is a well defined product system.  We now prove the decomposability of $E$.  
  
  For $b\in P$, let $D(b)$ and $D^{V}(b)$ be the  set of decomposable vectors in $E(b)$ and $E^{V}(b)$ respectively.  Let $u\in E^{V}(b)$. We claim that $u\in D^{V}(b)$ if and only if $u\in D(b)$. Suppose $u\in D^{V}(b)$. Let $a \in P$ be such that $a\leq b$. Let $u_{a}\in E^{V}(a)$  and $u_{a,b}\in E^{V}(b-a)$ be such that $u=u_{a}\odot u_{a,b}$. Note that $W(\Gamma(a,b))u_{a,b}\in E(b-a)$. Clearly,
\begin{equation*} u=u_{a}.\frac{1}{\alpha(a,b-a)}W(-\Gamma(a,b))u_{a,b}.
\end{equation*}
Therefore, $u$ is decomposable in $E$. It may similarly be proved that if $u\in D(b)$, then $u \in D^V(b)$. Hence the claim is proved. 

Since $E^{V}$ is decomposable, $D(a)$ is total in $E(a)$ for each $a\in P$. It follows from Eq. \ref{ccrdecomposable} and Eq. \ref{multiplication}  that $D(a)D(b)=D(a+b)$ for $a,b\in P$.
 This completes the proof.
  \end{proof}
  
  \begin{prop}
  \label{admissibility of cohomologous cocycles}
  Let $\Gamma_{1},\Gamma_{2}\in Z^{2}(P,\mathcal{H})$ be given. Suppose $\Gamma_{1}$ and $\Gamma_{2}$ are cohomologous. Then, $\Gamma_{1}$ is admissible if and only if $\Gamma_{2}$ is admissible.
  \end{prop}
  \begin{proof}
  Let $P\ni a\to\xi_{a}\in\mathcal{H}$ be a measurable map such that $\xi_{a}\in Ker(V_{a}^{*})$ for $a\in P$, and \[V_{a}(\Gamma_{2}-\Gamma_{1})(a,b)= \xi_{a+b}-\xi_{a}-V_{a}\xi_{b}\] for $a,b\in P$.
  
  Suppose $\Gamma_{1}$ is admissible. 
  Let $\alpha_{1}:P\times P\to\mathbb{T}$ be a Borel map such that  \[e^{iIm\langle\Gamma_{1}(a,b+c)|V_{b}\Gamma_{1}(b,c)\rangle}=\frac{\alpha_{1}(a,b)\alpha_{1}(a+b,c)}{\alpha_{1}(a,b+c)\alpha_{1}(b,c)}\]for $a,b,c\in P$.

 Consider the product system $E^{(\alpha_{1},\Gamma_{1},V)}$.
  For every $a\in P$, define a unitary operator
  $\Psi_{a}:E^{(\alpha_{1},\Gamma_{1},V)}(a)\to E^{(\alpha_{1},\Gamma_{1},V)}(a)$  by \[\Psi_{a}(u)=W(-\xi_{a})u\] for $u\in\Gamma_{s}(Ker(V_{a}^{*}))$.

  Let $F=\displaystyle \coprod_{a\in P}F(a)$ where $F(a)=\Gamma_{s}(Ker(V_{a}^{*}))$ for $a\in P$.  Define a product on $F$ by
  \[u\circ v=\Psi_{a+b}^{-1}(\Psi_{a}(u)*\Psi_{b}(v))\] for $a,b\in P$, $u\in F(a) $ and $v\in F(b)$, where $*$ is the product rule in $E^{(\alpha_1,\Gamma_1,V)}$, i.e.
\[u\circ v=\alpha_{1}(a,b)\ W(\xi_{a+b})(W(-\xi_{a})u\odot W(\Gamma_{1}(a,b))W(-\xi_{b})v)\] for $a,b\in P$, $u\in F(a) $ and $v\in F(b)$.  Here, $\odot$ is the product rule of the CCR flow $\alpha^{V}$. Clearly, $F$  with the above defined product is a product system.

Note that
\begin{equation}\label{alpha2}
e(\xi)\circ e(\eta)=\alpha_{2}(a,b)\ e(\xi)\odot W(\Gamma_{2}(a,b))e(\eta)\end{equation} for $a,b\in P$, $\xi\in Ker(V_{a}^{*})$ and $\eta\in Ker(V_{b}^{*})$.
 Here, \[\alpha_{2}(a,b)=\alpha_{1}(a,b)e^{-iIm(\langle\Gamma_{1}(a,b)|\xi_{b}\rangle+\langle\xi_{a+b}|-V_{a}\Gamma_1(a,b)+\xi_{a}+V_{a}\xi_{b}\rangle)}.\]  It is clear that $||\alpha_{2}(a,b)||=1$. Observe that $\alpha_2$ is measurable as well. By the associativity of the product $\circ$ and Eq. \ref{alpha2}, $\alpha_{2}$ must satisfy the identity 
\[e^{iIm\langle\Gamma_{2}(a,b+c)|V_{b}\Gamma_{2}(b,c)\rangle}=\frac{\alpha_{2}(a,b)\alpha_{2}(a+b,c)}{\alpha_{2}(a,b+c)\alpha_{2}(b,c)}\] for $a,b,c\in P$. 
Hence, $\Gamma_2$ is admissible. 
This completes the proof.
  \end{proof}
Let $E$ be a decomposable product system over $P$. For the rest of this section, the decomposable product system $E$ is fixed.  For $a \in \Omega$, let $D(a)$ be the set of decomposable vectors in $E(a)$.  Identify two vectors in $D(a)$ if they are scalar multiples of each other.  
Denote the set of equivalence classes by $\Delta(a)$. For $u \in D(a)$, the equivalence class in $\Delta(a)$ containing $u$ will be denoted $\dot{u}$.
 The disjoint union 
\begin{equation*}
\Delta:=\displaystyle{\coprod_{a\in \Omega}}\Delta(a)
\end{equation*}
is called the \emph{path space} of $E$ over $\Omega$. 

Define
\begin{equation*}
\Delta^{(2)}:=\{(a,\dot{u},\dot{v}):\textit{ $a\in \Omega$, $u,v\in D(a)$}\}.
\end{equation*}
By an abuse of notation, we will mostly denote the element $(a,\dot{u},\dot{v})$ of $\Delta^{(2)}$ by $(a,u,v)$. Similarly, an element of $\Delta$ will be denoted by $(a,u)$ for $u\in D(a)$ and $a\in\Omega$.
Arveson's notion of $e$-logarithm was extended for decomposable product systems over cones in \cite{SUNDAR}, and the following  is one of the main results of \cite{SUNDAR}. 

    Let $e:=(e_a)_{a \in P}$ be a left coherent section of decomposable vectors of unit norm. Then, there exists a unique map $L^{e}:\Delta^{(2)}\to\mathbb{C}$ with the following properties.  
\begin{enumerate}
\item For $a \in \Omega$ and $u,v \in D(a)$,
\begin{equation*}
e^{L^{e}(a,\dot{u},\dot{v})}=\frac{\langle u|v\rangle}{\langle u|e_{a}\rangle\langle e_{a}|v\rangle}.
\end{equation*}.
\item For two left coherent sections $(u_a)_{a \in P}$ and $(v_a)_{a \in P}$, the map 
\begin{equation*}
\Omega\ni a\to L^{e}(a,u_{a},v_{a})
\end{equation*} is continuous and $\displaystyle \lim_{a\to 0} L^{e}(a,u_{a},v_{a})=0$.
\item For each $a\in\Omega$, the map \[D(a)\times D(a)\ni (u,v)\to L^{e}(a,u,v)\] is positive definite.

\item For $a\in \Omega$, there exists a function $\psi_{a}:\Delta\to\mathbb{C}$ such that for $u_{1},v_{1}\in D(a)$ and $u_2,v_2 \in D(b)$, with $b \in \Omega$,
\begin{equation*}
L^{e}(a+b,u_{1}u_{2},v_{1}v_{2})=L^{e}(a,u_{1},v_{1})+L^{e}(b,u_{2},v_{2})+\psi_{a}(b,u_{2})+\overline{\psi_{a}(b,v_{2})}.
\end{equation*}
\end{enumerate}
The map $L^{e}$ is called the e-logarithm with respect to the left coherent section $e$.

\begin{remark}
Note that the $e$-logarithm, as of now, is only defined for the path space over $\Omega$, and it is not a priori clear how to extend this to the path space over $P$.
We will show later, after deriving the structure of a decomposable product system, that the $e$-logarithm can, in fact, be defined for the path space over $P$.
\end{remark}

\begin{prop}\label{existencemeasurablesection}
Let $b\in \Omega$, and let $u\in D(b)$ be of norm one. Then, there exists a   left coherent section $e=\{e_{a}\}_{a\in P}$ of decomposable vectors of norm one such that 
\begin{enumerate}[(1)]
\item the map $\Omega\ni a \to e_{a}\ni \displaystyle{\coprod_{a\in\Omega}}E(a)$ is measurable, and 
\item $e_{b}=u$.
\end{enumerate}
\end{prop}
\begin{proof}
By Prop. 3.3 of \cite{SUNDAR} there exists a left coherent section $(f_{a})_{a\in P}$ of decomposable vectors such that $f_{b}=u$. We may assume that $||f_{a}||=1$ for each $a\in P$. 

We claim that there exists a Borel map $\lambda:\Omega\to\mathbb{C}$ such that $|\lambda_{a}|=1$ for each $a\in \Omega$, and $(\lambda_{a}f_{a})_{a\in \Omega}$ is a measurable section.  Let $a_{0}\in \Omega$ be fixed. For $n\in\mathbb{N}$, let $a_{n}=na_{0}$. Define \begin{equation*}(0,a_{n}):=\{b\in \Omega:\textit{ $b<a_{n}$}\}.\end{equation*}
It suffices to  prove that for each $n\in\mathbb{N}$, there exists a measurable map $\lambda:(0,a_{n})\to\mathbb{C}$  such that  the map $(0,a_{n})\ni a\to\lambda_{a}f_{a} \in \displaystyle \coprod_{a \in P}E(a)$ is measurable.

Let $n \in \mathbb{N}$ be given. 
For $a\in (0,a_{n})$, let $L_{a}:E(a_{n}-a)\to E(a_{n})$ be the operator defined by \[L_a(\xi)= f_{a}\xi.\] Set $P_a=L_{a}L_{a}^{*}$. As in Prop. 6.6.2 of \cite{arveson}, the proof will follow once we show that $(0,a_n) \ni a\to P_{a}\in B(E(a_{n}))$ is measurable where $B(E(a_n))$ is given the $\sigma$-algebra generated by  the weak operator topology. We need to prove that for $u,v\in E(a_{n})$, the map $(0,a_n) \ni a\to \langle P_{a}u|v\rangle \in \bbc$ is measurable. 

For $a\in P$, let 
\begin{equation*}
D^{f}(a):=\{w\in D(a):\textit{ $\langle w| f_{a}\rangle=1$}\}. 
\end{equation*}
  Since $D^{f}(a_{n})$ is total in $E(a_{n})$, it is enough to prove that $(0,a_{n})\ni a\mapsto \langle P_{a}u|v\rangle\in\bbc$ is measurable for $u,v\in D^{f}(a_{n})$. Let $(u_{a})_{a\in P}$ and $(v_{a})_{a\in P}$ be left coherent sections of decomposable vectors such that $u_{a_{n}}=u$ and $v_{a_{n}}=v$. We may assume $u_{a},v_{a}\in D^{f}(a)$ for each $a\in P$.
Then, for $a \in (0,a_n)$,
\begin{equation*}
P_{a}(u)=\langle u_{a}|f_{a}\rangle f_{a}u(a,a_{n})=f_au(a,a_n).
\end{equation*}
This implies that
\[
\langle P_{a}u|v\rangle= \langle f_{a}|v_{a}\rangle\langle u(a,a_{n})|v(a,a_{n})\rangle=\frac{\langle u_{a_{n}}|v_{a_{n}}\rangle}{\langle u_{a}|v_{a}\rangle}=\frac{\langle u_{a_{n}}|v_{a_{n}}\rangle}{e^{L^{f}(a,u_a,v_a)}}.\]
 Since $(0,a_n) \ni a \to L^{f}(a,u_a,v_a) \in \bbc$ is continuous,  the map \[(0,a_{n})\ni a\to\langle P_{a}u|v\rangle \in \mathbb{C}\] is measurable. This completes the proof.
\end{proof}
\begin{remark}
\label{existencemeasurablesection1} With Prop. \ref{existencemeasurablesection} in hand, it is not difficult to show that there  exists a sequence $\{e_{n}\}_{n\in\mathbb{N}}$ of left coherent sections of decomposable vectors in $E$ such that
\begin{enumerate}[(1)]
\item for $n\in\mathbb{N}$, the section $e_{n}$ is measurable on $\Omega$, and
\item for each $a\in P$, the set $\{(e_{n})_{a}:\textit{ $n\in\mathbb{N}$}\}$ is total in $E(a)$.
\end{enumerate} 
This can be proved by proceeding exactly as in Lemma 6.7.5 of \cite{arveson}..
\end{remark}

Fix a left coherent section $e$ of decomposable vectors in $E$ such that the map \[\Omega\ni a\to e_{a}\in \displaystyle{\coprod_{b\in\Omega}}E(b)\] is measurable.
This is guaranteed by Prop. \ref{existencemeasurablesection}. Let $L^e$ be the $e$-logarithm w.r.t $e$. The left coherent section $e$ is fixed untill further mention.

We use the $e$-logarithm to construct a  pure isometric representation $V^{E}$ of $\Omega$ on a separable Hilbert space $\mathcal{H}$ as follows. We merely state the construction of $V^{E}$. More details can be found in pages 16-18 of \cite{SUNDAR}.
\begin{enumerate}[(1)]
\item For $a\in\Omega$, let \[\mathbb{C}_{0}\Delta(a):=\{f:\Delta(a) \to \mathbb{C}: \textrm{$f$ is finitely supported and $\sum_{u \in \Delta(a)}f(u)=0$}\}.\] Define a  semi-definite inner product on $\mathbb{C}_{0}\Delta(a)$ by
\begin{equation*}
\langle f|g\rangle=\sum_{u,v\in \Delta(a)}f(u)\overline{g(v)}L^{e}(a,u,v).
\end{equation*}
Let $H_{a}$ be the Hilbert space obtained from  $\langle ~|~\rangle$. For $u\in\Delta(a)$, let $\delta_{u}$ denote the indicator function at $u$. If $u_{1},u_{2}\in \Delta(a)$, the element $\delta_{u_{1}}-\delta_{u_{2}}\in H_{a}$ is denoted by $[u_{1}]-[u_{2}]$.
\item Let $a<b$.  Embed $H_a$ isometrically into $H_b$ via the map
\begin{equation*}
H_{a}\ni [u_{1}]-[u_{2}]\to [u_{1} v]-[u_{2}v]\in H_{b}
\end{equation*}
where $v$ is any element in $D(b-a)$. The embedding does not depend on the chosen element $v$.  The inductive limit of the Hilbert spaces $(H_{a})_{a \in \Omega}$ is denoted by $\mathcal{H}$. Then, the Hilbert space $\mathcal{H}$ is separable.
\item Let $a\in\Omega$, and let $v \in D(a)$. Let $V_a$ be the isometry on $\mathcal{H}$ defined by 
\begin{equation*}
V_a([u_{1}]-[u_{2}])=[v u_{1}]-[vu_{2}].
\end{equation*} The isometry $V_a$ does not depend on the chosen element $v$. Then,  $V^{E}=\{V_{a}\}_{a\in\Omega}$ is a pure isometric representation of $\Omega$ on $\mathcal{H}$. We call $V^E$ \emph{the isometric representation of $\Omega$ associated to the decomposable product system $E$}.
\end{enumerate}

Let $E|_{\Omega}$ denote the restriction of $E$ to $\Omega$. We will first describe the structure of $E|_{\Omega}$, and show that the same structure holds over $P$ as well, after going through a few extension results.  
 We will denote $V^{E}$ by $V$ for simplicity.  
 Define $\Gamma:\Omega\times\Omega\to\mathcal{H}$ by
\begin{equation*}
\Gamma(a,b)=[e_{b}]-[e(a,a+b)]
\end{equation*} for $a,b\in\Omega$.

We need the following version of Lemma 5.5.8 and Lemma 6.7.3 of \cite{arveson}.

\begin{lemma}\label{Gammameasurability} Keep the foregoing notation. Fix $b \in \Omega$. Then, the map \[\Omega \ni a\to \langle\Gamma(a,b)|\xi\rangle\in\mathbb{C}\] is measurable for every $\xi\in\mathcal{H}$.
\end{lemma}
\begin{proof}
 Since $\Gamma(a,b)\in Ker({V}_{b}^{*})$, it suffices to prove that
\begin{equation*}
a\to\langle[e_{b}]-[e(a,a+b)]|[u_{1}]-[u_{2}]\rangle\in\mathbb{C}
\end{equation*} is Borel for each $u_{1},u_{2}\in D(b)$.
Let $u_1,u_2 \in D(b)$ be given. Note that 
\begin{align*}%\label{eqtn}
\langle[e_{b}]-[e(a,a+b)]|[u_{1}]-[u_{2}]\rangle=&\langle[e_{b}]-[e(a,a+b)]|[u_{1}]-[u_{2}]\rangle\\
                    =&L^{e}(b,e_{b},u_{1})-L^{e}(b,e_{b},u_{2})\\
                    -&
                    L^{e}(b,e(a,a+b),u_{1})+L^{e}(b,e(a,a+b), u_{2}).
\end{align*}
  The term involving $a$ in RHS is $L^{e}(b,e(a,a+b),u_{1})-L^{e}(b,e(a,a+b), u_{2})$. Thus, it suffices to prove that given 
 $u \in D(b)$, the map 
$\Omega \ni a \to L^{e}(b,e(a,a+b),u) \in \mathbb{C}$
is measurable. 

Fix $u \in D(b)$. 
Let $a_{0}\in\Omega$ be fixed. Let $(0,a_{0})=\{a\in\Omega:a<a_{0}\}$. We show that the map
\begin{equation*}
(0,a_{0})\ni a \to L^{e}(b,e(a,a+b),u)\in\mathbb{C}
\end{equation*} is Borel. 

 For $n\in \mathbb{N}$, let $P_{n}=\{0=t_0<t_1<t_2<\cdots <t_{n}=1\}$ be a partition  of $[0,1]$ into $n$ subintervals $[t_{k},t_{k+1}]$. Assume that $|P_{n}|\to 0$ as $n\to\infty$. Let $\{u_{a}\}_{a\in P}$ be a left coherent section such that $u_{b}=u$. Let $n\in\mathbb{N}$. For each $k=0,1,2,\cdots , n-1$, let \begin{equation*}
u_{k}=\frac{u(t_{k}b,t_{k+1}b)}{\langle u(t_{k}b,t_{k+1}b)|e(t_{k}b,t_{k+1}b)\rangle}.
\end{equation*} 
Similarly, let 
\begin{equation*}
e_{k}(a)=\frac{e(a+t_{k}b,a+t_{k+1}b)}{\langle e(a+t_{k}b,a+t_{k+1}b)|e(t_{k}b,t_{k+1}b)\rangle}
\end{equation*}
for each $a\in\Omega$.
Fix $a\in\Omega$. 
By Remark 6.5.7 of \cite{arveson} and by Remark 3.14 of \cite{SUNDAR}, we have
\begin{equation*}
L^{e}(b,e(a,a+b),u)=\displaystyle{\lim_{n\to\infty}}\sum_{k=0}^{n-1}(\langle e_{k}(a)|u_{k}\rangle-1).
\end{equation*}
 Observe that
\begin{align*}\langle e_{a+t_{k+1}b}|e_{a+t_{k}b}u(t_{k}b,t_{k+1}b)\rangle=&\langle e_{a+t_{k}b}e(a+t_{k}b, a+t_{k+1b})|e_{a+t_{k}b}u(t_kb,t_{k+1}b)\rangle\\
                =&||e_{a+t_{k}b}||^{2}\langle e(a+t_{k}b,a+t_{k+1}b)|u(t_{k}b,t_{k+1}b)\rangle\\
                =&\langle e(a+t_{k}b,a+t_{k+1}b)|u(t_{k}b,t_{k+1}b)\rangle.\\
\end{align*}
Similarly, \[\langle e_{a+t_{k+1}b}|e_{a+t_{k}b}e(t_{k}b,t_{k+1}b)\rangle=\langle e(a+t_{k}b,a+t_{k+1}b)|e(t_{k}b,t_{k+1}b)\rangle\] for $k=1,2,\cdots ,n-1$.
Thus,
\begin{equation*}
\langle e_{k}(a)|u_{k}\rangle= \frac{\langle e_{a+t_{k+1}b}|e_{a+t_{k}b}u(t_{k}b,t_{k+1}b)\rangle}{\langle e_{a+t_{k+1}b}|e_{a+t_{k}b}e(t_{k}b,t_{k+1}b)\rangle\langle u(t_{k}b,t_{k+1}b)|e(t_{k}b,t_{k+1}b)\rangle }
\end{equation*} for each $k=0,1,2,\cdots n-1$.
Now the conclusion follows from the fact that $e=\{e_a\}_{a \in \Omega}$ is measurable. 
\end{proof}
Let $\alpha:\Omega\times\Omega\to\mathbb{T}$ be defined by 
\begin{equation*}
\alpha(a,b)=\frac{\langle e_{b}|e(a,a+b)\rangle}{|\langle e_{b}|e(a,a+b)\rangle|}
\end{equation*} for $a,b\in\Omega$.
\begin{lemma} With the foregoing notation, the map $\Gamma$ is  a 2-cocycle. The map $\alpha$ is Borel, and satisfies the equation
\[
\frac{\alpha(a,b)\alpha(a+b,c)}{\alpha(a,b+c)\alpha(b,c)}=e^{i Im \langle \Gamma(a,b+c)|V_b\Gamma(b,c)\rangle}\] for each $a,b,c\in\Omega$.
\end{lemma}
\begin{proof}
 Note that, for $b \in \Omega$, $Ker(V_{b}^{*})$ is spanned by  $\{[u_{1}]-[u_{2}]: u_{1},u_{2}\in D(b)\}$.
Hence, it is clear that $\Gamma(a,b)\in Ker(V_{b}^{*})$ for $a,b\in\Omega$. 

To see that $\Gamma$ is a $2$-cocycle, we repeat the calculation done in Lemma 5.5.7 of \cite{arveson}.
For $a,b,c\in\Omega$, calculate as follows to observe that
\begin{align*}
&V_{a+b}\Gamma(a+b,c)-V_{a}\Gamma(a,b+c)-V_{a+b}\Gamma(b,c)\\
&=([e_{a+b}e_{c}]-[e_{a+b+c}])-([e_{a}e_{b+c}]-[e_{a+b+c}])
-([e_{a}e_{b}e_{c}]-[e_{a}e_{b+c}])\\
                           &= [e_{a+b}e_{c}]-[e_{a}e_{b}e_{c}]\\
                           &=[e_{a+b}]-[e_{a}e_{b}]\\
                           &=-V_{a}\Gamma(a,b).
\end{align*}
Multiplying  the above equation by $V_{a}^{*}$, we see that $\Gamma$ satisfies the cocycle identity.

Fix $b_{0}\in\Omega$. For $n\in\mathbb{N}$, let $b_{n}=nb_{0}$. Let $b\in\Omega$. Let $N\in\mathbb{N}$ be such that $b<b_{N}$. By the cocycle identity, observe that 
\begin{equation}
\label{coherence of Gamma}
\Gamma(a,b)=E_{b}^{\perp}\Gamma(a,b_{N}).
\end{equation}
By Lemma \ref{Gammameasurability}, the map $\Omega \ni a \to \Gamma(a,b_N) \in \mathcal{H}$ is Borel. Eq. \ref{coherence of Gamma} implies that the map 
\[
\Omega \times \Omega \ni (a,b) \to \Gamma(a,b) \in \mathcal{H}\]
is measurable. 

Now we show the admissibility of $\Gamma$.
Calculate as follows to observe that for $a,b,c \in \Omega$,
\begin{align*}
e^{\langle\Gamma(a,b+c)|V_{b}\Gamma(b,c)\rangle}=&e^{\langle [e_{b+c}]-[e(a,a+b+c)]|[e_{b}e_{c}]-[e_{b+c}]\rangle}\\
=&e^{-L^{e}(b+c, e(a,a+b+c),e_{b}e_{c})}\\
=&\frac{\langle e(a,a+b+c)|e_{b+c}\rangle\langle e_{b+c}|e_{b}e_{c}\rangle}{\langle e(a,a+b+c)|e_{b}e_{c}\rangle}\\
=&\frac{\langle e(a,a+b+c)|e_{b+c}\rangle\langle e(b,b+c)|e_{c}\rangle}{\langle e(a,a+b+c)|e_{b}e_{c}\rangle}\\
=&\frac{\langle e(a,a+b+c)|e_{b+c}\rangle\langle e(b,b+c)|e_{c}\rangle}{\langle e(a,a+b)|e_b\rangle \langle e(a+b,a+b+c)|e_c\rangle}.
\end{align*}
Notice that for $a,b,c \in \Omega$, \begin{align*}
e^{iIm\langle\Gamma(a,b+c)|V_{b}\Gamma(b,c)\rangle}=&\frac{e^{\langle\Gamma(a,b+c)|V_{b}\Gamma(b,c)\rangle}}{|e^{\langle\Gamma(a,b+c)|V_{b}\Gamma(b,c)\rangle}|}
\end{align*}
It is now clear that
\begin{equation*}
e^{iIm\langle\Gamma(a,b+c)|V_{b}\Gamma(b,c)\rangle}=\frac{\alpha(a,b)\alpha(a+b,c)}{\alpha(a,b+c)\alpha(b,c)}
\end{equation*} for each $a,b,c\in\Omega$. Since $\langle e(a,a+b)|e_{b}\rangle=\langle e_{a+b}|e_{a}e_{b}\rangle$ and since $\Omega\ni a\to e_{a}\in E$ is measurable, $\alpha$ is Borel. Now the proof is complete.
 \end{proof}
 
As in Prop. \ref{->}, we may construct the product system  $E^{(\alpha,\Gamma,V)}$ over $\Omega$, i.e. for $a \in \Omega$, $E^{\alpha,\Gamma,V}(a)=\Gamma_s(Ker(V_a^{*}))$, and the product rule on the  exponential vectors is given by
\begin{equation*}
e(\xi).e(\eta)=\alpha(a,b) e(\xi)\odot W(\Gamma(a,b))e(\eta)
\end{equation*} where $\xi \in Ker(V_{a}^{*})$, $\eta\in Ker(V_{b}^{*})$ and $\odot$ is the multiplication rule of the product system of the CCR flow $\alpha^{V}$.
\begin{prop}\label{structure over omega}
With the foregoing notation, the restricted product system $E|_{\Omega}$ is isomorphic to $E^{(\alpha,\Gamma,V)}$.
\end{prop}
\begin{proof}
Let $F=E^{(\alpha,\Gamma,V)}$ for simplicity. 
Fix $a\in\Omega$. Define $\Psi_{a}:D(a)\to F(a)$ by
\begin{equation*}
\Psi_{a}(u)=\langle u|e_{a}\rangle e([u]-[e_{a}])
\end{equation*} for $u\in D(a)$.
It follows from the definition of $L^e$ that
\begin{equation*}
\langle\Psi_{a}(u)|\Psi_{a}(v)\rangle=\langle u|v\rangle
\end{equation*} for $u,v\in D(a)$.

Let $a\in\Omega$. Consider the one parameter path space 
 \[
 \widetilde{\Delta}=\coprod_{t>0}\widetilde{\Delta}(t)\]
 where $\widetilde{\Delta}(t):=\{(t,\dot{u}): u \in D(ta)\}$.
 
 Similarly, we may define $\widetilde{\Delta}^{(2)}$.
 Consider the e-logarithm $L^{e}$ of $E$. Restrict $L^{e}$ to $\widetilde{\Delta}^{2}$ which we denote by $\widetilde{L}^{e}$.  If we apply Arveson's method of constructing a $1$-parameter isometric representation, say $W:=\{W_t\}_{t \geq 0}$, out of the path space $(\widetilde{\Delta},\widetilde{L}^e)$  described in Chapter 5 of \cite{arveson}, we see that  $W_t=V_{ta}$ for every $t>0$. 
 
   By Thm. 5.5.5 of \cite{arveson}, there exists a measurable map $\xi:(0,\infty)\to\mathcal{H}$  such that 
  \begin{enumerate}
  \item[(1)] for $t>0$, $\xi_t \in Ker(V_{ta}^{*})$, and
  \item[(2)] for $s,t>0$,$
 V_{sa}\Gamma(sa,ta)=\xi_{s+t}-\xi_{s}-V_{sa}\xi_{t}$.
 \end{enumerate}
 Thanks to Thm. 5.5.1 of \cite{arveson}, $\{e([u]-[e_{a}]-\xi_{1}): u\in D(a)\}$ is total in $F(a)$. Consequently, its image under $W(\xi_1)$ is total in $F(a)$. Hence, $\{e([u]-[e_a]): u \in D(a)\}$ is total in $F(a)$.  Therefore, $\Psi_{a}(D(a))$ is total in $F(a)$. 

 Again, a routine computation verifies that
 \begin{equation*}
 \Psi_{a+b}=\Psi_{a}\Psi_{b}
\end{equation*} for $a,b\in\Omega$.

We now show that $\Psi:E|_{\Omega}\to F$ is measurable. Thanks to Remark \ref{existencemeasurablesection1}, it suffices to show that  for any left coherent measurable section $(u_{a})_{a\in P}$ of decomposable vectors, the map $\Omega\ni a\to \Psi_{a}(u_{a})\in \Gamma_{s}(\mathcal{H})$ is measurable.
Let $(u_{a})_{a\in\Omega}$ be a left coherent measurable section of decomposable vectors. For $b \in \Omega$, let 
\[
(0,b):=\{c \in \Omega: c<b\}.
\]
Fix $b \in \Omega$. It suffices to show that the map 
\[
(0,b) \ni a \to [u_a]-[e_a] \in \mathcal{H}
\]
is measurable. 

Note that for  $a \in (0,b)$, $[u_a]-[e_a] \in Ker(V_b^*)$ and   $\{[u_1]-[u_2]:v_1,v_2 \in D(b)\}$ is total in $Ker(V_b^*)$. Thus, it suffices to show that for $u_1,u_2 \in D(b)$, the map 
\[
(0,b) \ni a \to \langle [u_a]-[e_a]|[u_1]-[u_2]\rangle \in \mathbb{C} 
\]
is measurable. 

For $a \in \Omega$, let $\psi_a:\Delta \to \mathbb{C}$ be such that for $u_{1},v_{1}\in D(a)$ and $u_2,v_2 \in D(c)$, with $c \in \Omega$,
\begin{equation*}
L^{e}(a+c,u_{1}u_{2},v_{1}v_{2})=L^{e}(a,u_{1},v_{1})+L^{e}(c,u_{2},v_{2})+\psi_{a}(c,u_{2})+\overline{\psi_{a}(c,v_{2})}.
\end{equation*}

Let $v_1,v_2 \in D(b)$ be given.   Let $\{(v_{i})_{a}\}_{a\in P}$ be  left coherent sections such that $(v_{i})_{b}=v_{i}$ for $i=1,2$. For $a \in (0,b)$, choose  $v\in D(b-a)$ and calculate as follows to observe that 

\begin{align*}
&\langle [u_{a}]-[e_{a}]|[v_{1}]-[v_{2}]\rangle\\
&=\langle [u_{a}v]-[e_{a}v]|[v_{1}]-[v_{2}]\rangle\\
&= L^{e}(a,u_{a},(v_{1})_{a})+L^{e}(b-a,v,v_{1}(a,b))+\psi_{a}(b-a,v)+\overline{\psi_{a}(b-a,v_{1}(a,b))}\\
&-L^{e}(a,e_{a},(v_{1})_{a})-L^{e}(b-a, v,v_{1}(a,b) )-\psi_{a}(b-a, v)-\overline{\psi_a(b-a, v_{1}(a,b))}\\
&-L^{e}(a,u_{a},(v_{2})_{a})-L^{e}(b-a,v,v_{2}(a,b))-\psi_{a}(b-a,v)-\overline{\psi_{a}(b-a,v_{2}(a,b))}\\
&+L^{e}(a,e_{a},(v_{2})_{a})+L^{e}(b-a, v,v_{2}(a,b) )+\psi_{a}(b-a,v)+\overline{\psi_a(b-a, v_{2}(a,b))}\\
&= L^{e}(a,u_{a},(v_{1})_{a})-L^{e}(a,e_{a},(v_{1})_{a})-L^{e}(a,u_{a},(v_{2})_{a})+L^{e}(a,e_{a},(v_{2})_{a}).
\end{align*} Thanks to the continuity of $L^e$, the map 
\[
(0,b) \ni a \to \langle [u_a]-[e_a]|[v_1]-[v_2] \rangle \in \mathbb{C}\]
is continuous. The measurability of $\Psi$ follows. 
Now the proof is complete.
\end{proof}

Note that  that the isometric representation $V$ on $\mathcal{H}$ extends from $\Omega$ to $P$ in the following sense: 
There exists a strongly continuous (pure) isometric representation $\widetilde{V}$ of $P$ on $\clh$ such that
$
\widetilde{V}_{a}=V_{a}$
 for $a\in\Omega$. 
Extend the  2-cocycle $\Gamma$ over $\Omega$  to a 2-cocycle $\widetilde{\Gamma}$ over $P$. This is guaranteed by Corollary \ref{extension third section}. 
\begin{remark}
\label{equality of algebraic cocycles}
We collect a fact regarding semigroup cohomology that we need. Let $S$ be a subsemigroup of $\R^{d}$, and let $G$ be an abelian group. For each $n\in\bbn$, let $C^{n}$ denote the set of all maps from $S^n \to G$. For each $n\in\mathbb{N}$, let $\delta^{n}:C^{n}\to C^{n+1}$ be the map defined by
\begin{align*}\delta^{n}f(x_{1},x_{2},\cdots , x_{n},x_{n+1})=&f(x_{2},x_{3},\cdots ,x_{n+1})\\
+&\sum_{i=1}^{n}(-1)^{i}f(x_{1},\cdots x_{i-1},x_{i}+x_{i+1},x_{i+2},\cdots, x_{n+1})\\
+&(-1)^{n+1}f(x_{1},x_{2},\cdots ,x_{n})\end{align*} for $f\in C^{n}$. 

For each $n\in\mathbb{N}$, $\delta^{n+1}\circ \delta^{n}=0$. Set $Z^n_{alg}(S,G):=Ker(\delta^{n})$ and $B^{n}_{alg}(S,G):=Im(\delta^{n-1})$. 
Set 
\[
H^n_{alg}(S,G):=\frac{Z^{n}_{alg}(S,G)}{B^{n}_{alg}(S,G)}.\]

Then, the  restriction map \[Res:H_{alg}^{n}(\R^d,\mathbb{T}) \to H_{alg}^n(P,\mathbb{T})\] and the restriction map \[Res:H_{alg}^{n}(P,\mathbb{T}) \to H_{alg}^n(\Omega,\mathbb{T})\] are isomorphisms. For a proof, the reader is referred to Prop. 4.1 (Page 191) of \cite{cartan}.
\end{remark}

Consider the map $F:P\times P\times P \to \mathbb{T}$ defined by \[F (a,b,c)= e^{iIm\langle \widetilde{\Gamma}(a,b+c)|\widetilde{V}_{b}\widetilde{\Gamma}(b,c)\rangle}.\]
Since $\widetilde{\Gamma}:P\times P\to\mathcal{H}$ is norm continuous, the map $F$ is continuous. The fact that  $F|_{\Omega} \in Z^{3}_{alg}(\Omega,\bbt)$  and the fact that $F$ is continuous imply  that $F\in Z_{alg}^{3}(P,\mathbb{T})$. Since $F|_{\Omega}$ is a coboundary, by Remark \ref{equality of algebraic cocycles}, $F \in B^{3}_{alg}(P,\mathbb{T})$. 
 Thus, there exists a map $\widetilde{\alpha}:P\times P\to\mathbb{T}$ such that 
\begin{equation*}
e^{iIm\langle\widetilde{\Gamma}(a,b+c)|\widetilde{V}_{b}\widetilde{\Gamma}(b,c)\rangle}=\frac{\widetilde{\alpha}(a,b)\widetilde{\alpha}(a+b,c)}{\widetilde{\alpha}(a,b+c)\widetilde{\alpha}(b,c)}
\end{equation*} for $a,b,c\in P$. Again, using Remark \ref{equality of algebraic cocycles}, we can choose $\widetilde{\alpha}$ in such a way that for $a,b \in \Omega$, $\widetilde{\alpha}(a,b)=\alpha(a,b)$.

 Let $\widetilde{\alpha}$, $\widetilde{\Gamma}$ and $\widetilde{V}$ be as above. We can construct an algebraic product system $E^{(\widetilde{\alpha},\widetilde{\Gamma},\widetilde{V})}$ using the triple $(\widetilde{\alpha},\widetilde{\Gamma},\widetilde{V})$ . We let $E^{(\widetilde{\alpha},\widetilde{\Gamma},\widetilde{V})}=\{\Gamma_{s}(Ker(\widetilde{V}_{a}^{*}))\}_{a\in P}$, and define the multiplication on the exponential vectors by
 \[e(\xi)e(\eta)=\widetilde{\alpha}(a,b)e(\xi)\odot W(\Gamma(a,b))e(\eta)\] for $a,b\in P$, $\xi\in Ker(\widetilde{V}_{a}^{*})$ and $\eta\in Ker(V_{b}^{*})$. For simplicity, we denote  $E^{(\widetilde{\alpha},\widetilde{\Gamma},\widetilde{V})}$ by $\widetilde{E}$. 
 
 \begin{remark}
 \label{same decom vectors}
For $a\in P$, let $\widetilde{D}(a)$ denote the set of decomposable vectors in $\widetilde{E}(a)$. Note that $\widetilde{E}$ and $E^{\widetilde{V}}$ (the product system of the CCR flow $\alpha^{\widetilde{V}}$) have the same set of decomposable vectors. 
 Therefore, for $a \in P$, \[\widetilde{D}(a)=\{\lambda e(\xi): \lambda \in \bbc\setminus\{0\},  \xi \in Ker(\widetilde{V}_a^*)\}.\]
 It can be proved as in Prop. \ref{->} that $\widetilde{E}$ is decomposable. 
\end{remark}

Keep the foregoing notation.
 \begin{prop}
 \label{iso with algebraic}    
  The algebraic product systems $E$ and $\widetilde{E}$ are isomorphic.
 \end{prop}
 \begin{proof} We have already shown that $E$ and $\widetilde{E}$ are isomorphic over $\Omega$.
 Both $E$ and $\widetilde{E}$ possess left coherent sections. By Corollary \ref{over omega implies over P1}, $E$ and $\widetilde{E}$ are isomorphic over $P$, as algebraic product systems.
 \end{proof}

  We are now in a position to prove that the $e$-logarithm can be defined for the path space of $E$  over $P$. 
For $a\in P$, let $D(a)$ denote the set of decomposable vectors in $E(a)$. We identify vectors in $D(a)$ if they are scalar multiples of each other. For $u\in D(a)$, let $\dot{u}$ denote its equivalence class with respect to this equivalence relation, and denote the set of equivalence classes by $\Delta(a)$. Define \[\Delta:=\{(a,\dot{u}):\textrm{ $a\in P$ and $\dot{u} \in \Delta(a)$}\}.\]
Let \[\Delta^{(2)}:=\{(a,\dot{u},\dot{v}):\textrm{$a\in P$, $u,v\in D(a)$}\}.\] Let $f=\{f_{a}\}_{a\in P}$ be a left coherent section of $E$ of unit norm.
 
 \begin{theorem}\label{existence e log over P} There exists a unique  map $L^{f}:\Delta^{(2)}\to\mathbb{C}$ satisfying the following properties.
 \begin{enumerate}[(1)]
 \item For $a\in P$, $u,v\in D(a)$,\[e^{L^{f}(a,\dot{u},\dot{v})}=\frac{\langle u|v\rangle}{\langle u|f_{a}\rangle\langle f_{a}|v\rangle}.\]
 \item For left coherent sections $u=\{u_{a}\}_{a\in P}$ and $v=\{v_{a}\}_{a\in P}$ of decomposable vectors in $E$, the map
 \[P\ni a \to L^{f}(a,\dot{u}_{a},\dot{v}_{a})\in\mathbb{C}\] is continuous, and
 $\displaystyle \lim_{a \to 0}L^{f}(a,\dot{u}_{a},\dot{v}_{a})=0$.
 \item For $a\in P$, the map \[D(a)\times D(a)\ni (u,v)\to L^{f}(a,\dot{u},\dot{v}) \in \mathbb{C}\] is positive definite.
 \item For each $a\in P$, there exists a map $\psi_{a}:\Delta\to\mathbb{C}$ such that
 \[L^{f}(a+b, \dot{{u}_{1}}\dot{v_{1}}, \dot{{u}_{2}}\dot{v_{2}})= L^{f}(a,\dot{u}_{1},\dot{u}_{2})+ L^{f}(b,\dot{v}_{1},\dot{v}_{2})+\psi_{a}(b,\dot{v}_{1})+\overline{\psi_{a}(b,\dot{v}_{2})}\] for $b\in P$, $u_{i}\in D(a)$, and $v_{i}\in D(b)$, for $i=1,2$.
 \end{enumerate}
 \end{theorem}
 \begin{proof} Since the assertion is purely algebraic, we may assume, thanks to Prop. \ref{iso with algebraic}, that $E=\widetilde{E}=E^{(\widetilde{\alpha},\widetilde{\Gamma},\widetilde{V})}$. Let $\widetilde{\Delta}$ be the path space of $\widetilde{E}$ over $P$. For $a \in P$, let $\widetilde{D}(a)$ be the set of decomposable vectors of $\widetilde{E}(a)$. Suppose $f:=(f_a)_{a \in P}$ is a left coherent section of $\widetilde{E}$ of unit norm. 
 
By Remark \ref{same decom vectors}, there exist scalars $\{\lambda_{a}\}_{a\in P}$ and a family $\{\xi_{a}\}_{a \in P}$ such that for $a \in P$, $\xi_a \in Ker(\widetilde{V}_a^*)$ and $f_a=\lambda_ae(\xi_a)$. The fact that $(f_a)_{a \in P}$ is left coherent implies that 
 if $a\leq b$, $\xi_{a}=(1-\widetilde{V}_{a}\widetilde{V}_{a}^{*})\xi_{b}$. The fact that $||f_a||=1$ implies that $|\lambda_a|=e^{-\frac{||\xi_a||^{2}}{2}}$ for $a \in P$. For $a, b \in P$, let $\xi(a,a+b)\in Ker(\widetilde{V}_{b}^{*})$ be such that \[\xi_{a+b}=\xi_{a}+\widetilde{V}_{a}(\widetilde{\Gamma}(a,b)+\xi(a,a+b)).\]

  Let $a\in P$, and let $u,v\in \widetilde{D}(a)$. There exist scalars $\lambda,\mu$ and $\xi,\eta\in Ker(\widetilde{V}_{a}^{*})$ such that $u=\lambda e(\xi)$ and $v=\mu e(\eta)$ respectively.
 We define \[L^{f}(a,\dot{u},\dot{v})=\langle \xi-\xi_{a}|\eta-\xi_{a}\rangle.\]
 It is clear that $L^{f}$ is well-defined, and satisfies $(1)$ and $(3)$. 

 Suppose $u=\{u_{a}\}_{a\in P}$ and $v=\{v_{a}\}_{a\in P}$  are two left coherent sections in $\widetilde{E}$, where $u_{a}=\delta_{a}e(\zeta_{a})$ and $v_{a}=\mu_{a}e(\eta_{a})$ for $a\in P$. Then, for $a\leq b$, $\zeta_{a}=(1-\widetilde{V}_{a}\widetilde{V}_{a}^{*})\zeta_{b}$ and $\eta_{a}=(1-\widetilde{V}_{a}\widetilde{V}_{a}^{*})\eta_{b}$.  By Remark \ref{continuity of coherence}, the map $P \ni a \to (\zeta_a,\eta_a,\xi_a) \in \mathcal{H}\times \mathcal{H} \times \mathcal{H}$ is continuous and $(\zeta_a,\eta_a,\xi_a) \to (0,0,0)$ as $a \to 0$. Now the second assertion is clear. 

Given $b\in P$, and $u\in D(b)$, there exists a left coherent section $\{u_{a}\}_{a\in P}$ in $\widetilde{E}$ such that $u=u_{b}$. Now using $(1)$ and $(2)$, the uniqueness of $L^{f}$ can be deduced.

Fix $a\in P$. We define $\psi:\Delta\to\mathbb{C}$ by
\[\psi_{a}(b,\mu e(\eta))=\langle\eta-\xi_{b}|\xi_{b}-\xi(a,a+b)\rangle+\frac{||\xi_{b}-\xi(a,a+b)||^{2}}{2}\] for $b\in P$.
 
Let $a,b\in P$. Let $u_{i}=\lambda_{i}e(\zeta_{i})$, $v_{i}=\mu_{i}e(\eta_{i})$, where $\zeta_{i}\in Ker(\widetilde{V}_{a}^{*})$ and $\eta_{i}\in Ker(\widetilde{V}_{b}^{*})$ for $i=1,2$. It is routine to verify that
\[L^{f}(a+b, \dot{{u}_{1}}\dot{v_{1}}, \dot{{u}_{2}}\dot{v_{2}})= L^{f}(a,\dot{u}_{1},\dot{u}_{2})+ L^{f}(b,\dot{v}_{1},\dot{v}_{2})+\psi_{a}(b,\dot{v}_{1})+\overline{\psi_{a}(b,\dot{v}_{2})}.\]
Now the proof is complete.
 \end{proof}

With the $e$-logarithm constructed for the path space of $E$ over $P$, we  construct a separable Hilbert space $\mathcal{H}$, and define a pure(strongly continuous) isometric representation  of $P$ on $\mathcal{H}$ just as we have done before. Simply replacing $\Omega$ by $P$ in all the steps required  to get the structure of $E|_{\Omega}$, we arrive at
the following structure theorem for a generic product system over $P$. We leave  
the verification that the entire proof works over $P$ to the reader.

\begin{theorem}\label{structure over P} Let $E$ be a decomposable product system over $P$, and let $V$ be 
the associated isometric representation of $E$ acting on $\mathcal{H}$. Then, there exist $\Gamma\in Z_{a}^{2}(P,\mathcal{H},V)$ and a Borel map $\alpha:P\times P\to\mathbb{T}$ satisfying \[e^{iIm\langle \Gamma(a,b+c)|V_{b}\Gamma(b,c)\rangle}=\frac{\alpha(a,b)\alpha(a+b,c)}{\alpha(a,b+c)\alpha(b,c)}\] for $a,b,c\in P$, such that the product systems $E$ and $E^{(\alpha,\Gamma,V)}$ are isomorphic.
\end{theorem}

\begin{remark}
Given a decomposable product system $E$ over $P$, a priori, we are unable to construct the $e$-logarithm for the path space of $E$ over $P$. It was only constructed for the path space over $\Omega$. 
              Then, an entire machinery needs to be developed to get the structure of $E|_{\Omega}$ from which it follows that $E$ is algebraically isomorphic to an algebraic product system
               for which the construction of $e$-logarithm is easy. Is it possible to directly construct the $e$-logarithm for the path space of $E$ over $P$ without going 
               through the intermediate step of deriving the structure of $E|_{\Omega}$? This will avoid repeating the analysis twice, once for $E|_{\Omega}$, and then for $E|_{P}$. 
               Note that this subtlety does not arise in the $1$-parameter situation, i.e. when $P=[0,\infty)$. 
\end{remark}

We point out another pedantic subtlety that might be of interest to the reader
in the following remark.
\begin{remark}
\label{pedantic remark}
First, note that the notion of $E_0$-semigroups, product systems and decomposable product systems make perfect sense
               for semigroups (at least abelian) other than cones. For the reader's convenience, let us define what we mean by a decomposable vector.  Let $G$ be an abelian group,
               and let $Q \subset G$ be a subsemigroup, not necessarily containing the origin. For $x,y \in G$, write $x \leq y$ if $y-x \in Q$.  Let $E:=\{E(a)\}_{a \in Q}$ be a product
               system over $Q$. A non-zero vector $u \in E(a)$ is said to be decomposable if given $b \leq a$, there exist $v \in E(b)$ and $w \in E(a-b)$ such that $u=vw$. 
               It is now clear what a decomposable product system over $Q$ means. An $E_0$-semigroup over $Q$ is said to be decomposable if the associated product system
               is decomposable. Fix an infinite dimensional, separable Hilbert space $\mathcal{H}$. The set of $E_0$-semigroups over $Q$ acting on $B(\mathcal{H})$ will be denoted by $\mathcal{E}(Q,\mathcal{H})$,
               and the set of decomposable $E_0$-semigroups over $Q$ on $B(\mathcal{H})$ will be denoted by $\mathcal{D}(Q,\mathcal{H})$. 
               
 Why do we care so much about whether the indexing semigroup is $P$ or $\Omega$? 
In the $1$-parameter theory of $E_0$-semigroups, the indexing semigroup is either $(0,\infty)$ or $[0,\infty)$, and it is clear from the start that there is no difference whether we choose 
to work with the closed half-line or the open half-line. %This is not so obvious in the multiparameter case. 

As far as the multiparameter case is concerned this is not so clear. For example, when one wishes to study two parameter theory of $E_0$-semigroups,
two obvious choices of the indexing semigroup present themselves. One is the closed quarter plane $[0,\infty)\times [0,\infty)$, and the other is the open quarter plane $(0,\infty)\times (0,\infty)$. (The closed quarter plane is better for aesthetic reasons as it allows us
to think of a $2$-parameter $E_0$-semigroup as two $1$-parameter semigroups that satisfy a commutation relation.) Both the choices are perfectly
reasonable, and it is not immediately obvious that both choices lead to the same two parameter theory of $E_0$-semigroups. Thus, it is of some importance to establish  that the two theories are consistent. In fact, we have that the set of $E_0$-semigroups over $[0,\infty)\times [0,\infty)$ and 
the set of $E_0$-semigroups over $(0,\infty)\times (0,\infty)$ are in bijective correspondence. 

More generally, suppose $P$ is a closed convex, spanning, pointed cone in $\bbr^d$, and suppose $\Omega$ is its interior. Then, the map, denoted $\Phi$,
\[
\mathcal{E}(P,\mathcal{H}) \ni \beta \to \beta|_{\Omega} \in \mathcal{E}(\Omega,\mathcal{H})\]
is a bijection. Moreover, $\Phi$ and $\Phi^{-1}$ preserve cocycle conjugacy. Although, this is not hard to prove, this fact is not completely obvious as in the $1$-parameter case.
A proof, for a more general semigroup, can be found in \cite{Murugan_measurable}.

Is $\Phi(\mathcal{D}(P,\mathcal{H}))=\mathcal{D}(\Omega,\mathcal{H})$? Once again,  answer to this question is completely trivial  in the $1$-parameter case as we do not see any marked difference
between the order $\leq$ and the order $<$. In the higher dimensional case, we do not know of any elementary argument that shows that $\Phi(\mathcal{D}(P,\mathcal{H}))=\mathcal{D}(\Omega,\mathcal{H})$.
The only proof that the authors know is to build the entire technical machinery required to  derive the structure of a typical element of $\mathcal{D}(P,\mathcal{H})$ and the structure of a typical element of $\mathcal{D}(\Omega,\mathcal{H})$ independently,
and compare them. A proof is sketched below.   

Let $\beta \in \mathcal{D}(P,\mathcal{H})$ be given, and let $E_\beta$ be the associated product system over $P$. Then, $E^{\beta} \cong E^{(\alpha,\Gamma,V)}$ for some triple $(\alpha,\Gamma,V)$ that is defined over $P$. It is routine to prove, as in Prop. \ref{->},  that $E^{(\alpha,\Gamma,V)}|_{\Omega} \in \mathcal{D}(\Omega,\mathcal{H})$. 

Conversely, let $\beta \in \mathcal{D}(\Omega,\mathcal{H})$ be given, and let $\widetilde{\beta} \in \mathcal{E}(P,\mathcal{H})$ be such that $\widetilde{\beta}|_{\Omega}=\beta$. Let $E_{\beta}$ be the product system of $\beta$ over $\Omega$,
and let $E_{\widetilde{\beta}}$ be the product system of $\widetilde{\beta}$ over $P$. First, we construct the $e$-logarithm for the path space of $E_{\beta}$ over $\Omega$ (note here that
the order we use is $<$). This can be achieved by imitating the techniques of \cite{SUNDAR}. 

Once the $e$-logarithm is constructed, we can argue as we have done so far in this paper to prove that $E_{\beta}$ is isomorphic to $E^{(\alpha,\Gamma,V)}$, 
where the triple $(\alpha,\Gamma,V)$ is defined only over $\Omega$. Now, extend the triple $(\alpha,\Gamma,V)$, algebraically, to $(\widetilde{\alpha},\widetilde{\Gamma},\widetilde{V})$ over $P$, and consider the algebraic product system $E^{(\widetilde{\alpha},\widetilde{\Gamma},\widetilde{V})}$ over $P$. The product system $E_{\widetilde{\beta}}$ is then algebraically isomorphic to $E^{(\widetilde{\alpha},\widetilde{\Gamma},\widetilde{V})}$.  Using the fact
that $E_{\widetilde{\beta}}$ and $E^{(\widetilde{\alpha},\widetilde{\Gamma},\widetilde{V})}$ are algebraically isomorphic, and observing that the notion of decomposability is purely algebraic, it is routine to show, as in Prop. \ref{->}, that $E_{\widetilde{\beta}} \in \mathcal{D}(P,\mathcal{H})$. Hence, $\Phi(\mathcal{D}(P,\mathcal{H}))=\mathcal{D}(\Omega,\mathcal{H})$. 

Is there an elementary argument, i.e from first principles, that establishes the equality $\Phi(\mathcal{D}(P,\mathcal{H}))=\mathcal{D}(\Omega,\mathcal{H})$? Of course, these are pedantic issues.
   \end{remark}           

For $i=1,2$, let $V^{(i)}$ be a pure isometric representation of $P$ on a separable Hilbert space $\mathcal{H}_{i}$. Let $\Gamma_{i}\in Z_{a}^{2}(P,\mathcal{H}_{i})$ and let $\alpha_{i}:P \times P \to \mathbb{T}$ be  a map satisfying the admissibility condition \[e^{iIm\langle \Gamma_{i}(a,b+c)|V^{(i)}_{b}\Gamma_{i}(b,c)\rangle}=\frac{\alpha_{i}(a,b)\alpha_{i}(a+b,c)}{\alpha_{i}(a,b+c)\alpha_{i}(b,c)}\]
for $a,b,c\in P$, and for $i=1,2$. 

\begin{theorem}\label{projectively isomorphic product systems} Keep the foregoing notation. 
 The product system $E^{(\alpha_{1},\Gamma_{1},V^{(1)})}$ is projectively isomorphic to $E^{(\alpha_{2},\Gamma_{2},V^{(2)})}$ if and only if there exists a unitary $U:\mathcal{H}_{1}\to\mathcal{H}_{2}$  such that \begin{enumerate}[(1)]
 \item $V_a^{(2)}=UV_{a}^{(1)}U^{*}$ for $a\in P$, and
 \item $\Gamma_{2}-U\Gamma_{1}$ is a coboundary.
 \end{enumerate}
 \end{theorem}
 \begin{proof} For simplicity, write $E^{(\alpha_{i},\Gamma_{i},V^{(i)})}=E_{i}$ for $i=1,2$. Suppose $E_{1}$ and $E_{2}$ are projectively isomorphic. For $a \in P$, denote the set of decomposable vectors of $E_i(a)$ by $D_i(a)$.
 Let $\Psi:E_1\to E_2$ be a projective isomorphism. Suppose $\omega:P\times P\to\mathbb{T}$ is a multiplier such that
 \begin{equation*}
\Psi_{a+b}=\omega(a,b)\Psi_{a}\Psi_{b}
\end{equation*} for $a,b\in P$.
Fix $a \in P$. Then, $\Psi_a(D_1(a))=D_2(a)$. Recall that \[D_i(a)=\{\lambda e(\xi): \lambda \in \mathbb{C}\setminus \{0\}, \xi \in Ker(V_a^{(i)*})\}.\]

Let $a\in\Omega$. Then, the Hilbert spaces $Ker(V^{(i)*}_{a})$ have the same dimension for $i=1,2$, since they are both separable and infinite dimensional. By Lemma 2.1 of \cite{guichardetsymmetric}, there exists a unitary $U_{a}:Ker(V^{(1)*}_{a})\to Ker(V^{(2)*}_{a})$,  a scalar $\lambda_{a}\in\R$, and $\xi_{a}\in Ker(V^{(2)*}_a)$   such that
 \begin{equation*}
 \Psi_{a}=e^{i \lambda_{a}} W(\xi_{a})\Gamma(U_{a}).
 \end{equation*} 

Using the identity
\begin{equation*}%\label{proj. iso.}
\Psi_{a+b}(e(\xi)e(\eta))=\omega(a,b)\Psi_{a}(e(\xi))\Psi_{b}(e(\eta)),
\end{equation*}  observe that 
\begin{align}\label{proj iso condition}
U_{a}\xi+\xi_{a}+V^{(2)}_{a}(\Gamma_{2}(a,b)+U_{b}\eta+\xi_{b})=U_{a+b}(\xi+V^{(1)}_{a}(\Gamma_{1}(a,b)+\eta))+\xi_{a+b}
\end{align}for $a,b\in \Omega$, $\xi\in Ker(V^{(1)*}_{a})$, and $\eta\in Ker((V^{(1)*}_{b})$. Set $\xi=\eta=0$ in Eq. \ref{proj iso condition}. Then, \begin{equation}
\label{projiso1}V^{(2)}_{a}\Gamma_{2}(a,b)+\xi_{a}+V^{(2)}_{a}\xi_{b}=U_{a+b}V^{(1)}_{a}\Gamma_{1}(a,b)+\xi_{a+b}
\end{equation}for $a,b\in \Omega$.  Setting $\eta=0$ in Eq. \ref{proj iso condition}, and using Eq. \ref{projiso1}, we see that for $a,b \in \Omega$ and $\xi \in Ker(V^{(1)*}_a)$, $U_{a+b}\xi=U_{a}\xi$. Hence, the unitaries $\{U_{a}\}_{a \in \Omega}$ patch up to give a single unitary $U:\mathcal{H}_{1}\to\mathcal{H}_{2}$.

Now, 
\begin{align}\label{proj iso condition2}
U\xi+\xi_{a}+V^{(2)}_{a}(\Gamma_{2}(a,b)+U\eta+\xi_{b})=U(\xi+V^{(1)}_{a}(\Gamma_{1}(a,b)+\eta))+\xi_{a+b}
\end{align}for $a,b\in \Omega$, $\xi\in Ker(V^{(1)*}_{a})$, and $\eta\in Ker(V^{(1)*}_{b})$. 
Setting $\xi=0$ in Eq. \ref{proj iso condition2}, and using Eq. \ref{projiso1}, we see that $U$ intertwines $V^{(1)}$ and $V^{(2)}$, i.e. $V^{(2)}_{a}=UV^{(1)}_{a}U^{*}$ for $a\in \Omega$. Using the strong continuity of $V^{(1)}$ and $V^{(2)}$, we may conclude that $V^{(2)}_{a}=UV^{(1)}_{a}U^{*}$ for $a\in P$.

Again, setting $\xi=\eta=0$ in Eq. \ref{proj iso condition2}, we deduce that
\[V^{(2)}_{a}\Gamma_{2}(a,b)-V_{a}^{(2)}U\Gamma_{1}(a,b)=\xi_{a+b}-\xi_{a}-V^{(2)}_{a}\xi_{b}\] for $a,b\in \Omega$. 

From the above equation, it is clear that $\xi_{a}=(1-V^{(2)}_{a}V_{a}^{(2)*})\xi_{a+b}$, i.e the vectors $\{\xi_{a}\}_{a\in \Omega}$ are coherent. By Remark \ref{continuity of coherence},  the map $\Omega\ni a\to\xi_{a}\in\mathcal{H}_{2}$ is measurable. Thus, when restricted to $\Omega$, $\Gamma_{2}-U\Gamma_{1}\in B^{2}(\Omega, \mathcal{H}_{2})$. By Corollary \ref{extension third section}, $\Gamma_{2}-U\Gamma_{1}\in B^{2}(P, \mathcal{H}_{2})$.

Conversely, suppose there exists a unitary $U:\mathcal{H}_{1}\to\mathcal{H}_{2}$  intertwining $V^{(1)}$ and  $V^{(2)}$ such that
 $\Gamma_{2}-U\Gamma_{1}$ is a coboundary. Let $\xi:P\to \mathcal{H}_{2}$ be a measurable map such that 
\begin{equation*}
V^{(2)}_{a}(\Gamma_{2}-U\Gamma_{1})(a,b)=\xi_{a+b}-\xi_{a}-V_{a}\xi_{b}
\end{equation*} for $a,b\in P$.

The operator $\Lambda_{a}:=W(\xi_{a})\Gamma(U)$ defines a unitary between the Hilbert spaces $E_{1}(a)$ and $E_{2}(a)$ for each $a\in P$. The maps $(\Lambda_{a})_{a\in P}$ induce a product structure on $\displaystyle{\coprod_{a\in P}}E_{1}(a)$ given by 
\[
u \circ v=\Lambda_{a+b}^{-1}(\Lambda_a(u)\Lambda_b(v)).\] Let $\beta:P\times P\to\bbt$ be defined by
\[\beta(a,b)=\alpha_{2}(a,b)e^{iIm(\langle\xi_{a+b}|\xi_{a}+V^{(2)}_{a}\Gamma_{2}(a,b)+\xi_{b}\rangle+\langle\Gamma_{2}(a,b)|\xi_{b}\rangle)}\] for $a,b\in P$.
Then, the multiplication rule $\circ$ is given by 
\begin{equation*}
e(\xi)\circ e(\eta)=\beta(a,b) e(\xi)\odot W(\Gamma_{1}(a,b))e(\eta)
\end{equation*} for $a,b\in P$, $\xi \in Ker(V^{(1)*}_a)$, and $\eta \in Ker(V^{(1)*}_b)$.  The associativity of the product $\circ$ implies that  \[e^{iIm\langle \Gamma_{1}(a,b+c)|V^{(1)}_{b}\Gamma_{1}(b,c)\rangle}=\frac{\beta(a,b)\beta(a+b,c)}{\beta(a,b+c)\beta(b,c)}\] for $a,b,c\in P$.

Thus, the map $\Lambda=(\Lambda_a)_{a\in P}$ is an isomorphism between $E^{(\beta,\Gamma_1,V^{(1)})}$ and $E^{(\alpha_2,\Gamma_2,V^{(2)})}$. Since $\alpha_1$ and $\beta$ are cohomologous, $E^{(\alpha_1,\Gamma_1,V^{(1)})}$ and $E^{(\alpha_2,\Gamma_2,V^{(2)})}$ are projectively isomorphic. 
 \end{proof} 

Thm. \ref{structure over P} and Thm. \ref{projectively isomorphic product systems}
can be succinctly expressed as an equality. For a pure isometric representation $V$
of $P$ on a Hilbert space $\mathcal{H}$, let $\mathcal{D}_V(P)$ be the collection (up to projective isomorphism) of
decomposable product systems whose associated isometric representation is $V$.
Then,
\begin{equation}
\label{succinct}
\mathcal{D}_V(P)=\frac{H^{2}_{a}(P,\mathcal{H})}{\mathcal{U}(M_V)}=\frac{H^1_{a}(P,\mathcal{L}_V)}{\mathcal{U}(M_V)}.
\end{equation}
Here, $M_V:=\{V_x,V_{x}^{*}:x \in P\}^{'}$, and $\mathcal{U}(M_V)$ denotes the unitary group of the von Neumann algebra $M_V$. 
Note that $\mathcal{U}(M_V)$ acts naturally on $H^2(P,\mathcal{H})$, and leaves $H_{a}^{2}(P,\mathcal{H})$ invariant. 
Similarly, $\mathcal{U}(M_V)$ acts naturally on $H^1(P,\mathcal{L}_V)$, and leaves $H^1_{a}(P,\mathcal{L}_V)$ invariant. Also, the isomorphism $H^2(P,\mathcal{H}) \cong H^1(P,\mathcal{L}_V)$ is $\mathcal{U}(M_V)$ equivariant.

\begin{remark}\label{condition for isomorphism} 
Arguing as in Thm. \ref{projectively isomorphic product systems}, we can prove that the following are equivalent.
\begin{enumerate}
\item[(1)] The product systems $E^{(\alpha_1,\Gamma_1,V^{(1)})}$ and $E^{(\alpha_2,\Gamma_2,V^{(2)})}$ are isomorphic.
\item[(2)] There exist a unitary $U:\mathcal{H}_{1}\to\mathcal{H}_{2}$, a measurable map $P\ni a\to \xi_{a}\in\mathcal{H}_{2}$ such that $\xi_{a}\in Ker(V^{(2)*}_{a})$ for each $a\in P$, and a Borel map $P\ni a\to \lambda_a \in \mathbb{T}$ such that
\begin{enumerate}[(a)]
\item for $a\in P$, $UV^{(1)}_{a}U^{*}=V^{(2)}_{a}$,
\item for $a,b\in P$, $V^{(2)}_{a}(\Gamma_{2}(a,b)-U\Gamma_{1}(a,b))=\xi_{a+b}-\xi_{a}-V^{(2)}_{a}\xi_{b}$, and
\item for $a,b\in P$, \[e^{iIm(\langle\xi_{a+b}|UV^{(1)}_{a}\Gamma_{1}(a,b)\rangle-\langle\Gamma_{2}(a,b)|\xi_{b}\rangle)}=\frac{\lambda_{a}\lambda_{b}}{\lambda_{a+b}}\frac{\alpha_{2}(a,b)}{\alpha_{1}(a,b)}.\] 
\end{enumerate}
Here, $\mathcal{H}_i$ is the Hilbert space on which $V^{(i)}$ acts, for $i=1,2$.
\end{enumerate}

Since we do not  know  how Condition $(c)$ is to be interpreted in the language of cohomology, we have chosen to work with   projective isomorphism classes rather than isomorphism classes.

\end{remark}

\section{Computation of cocycles}
In this section, we compute the space of cocycles and the space of admissible cocycles for the shift semigroup associated to a transitive action of $P$. 
We need to work in a slightly general setting to get a few initial results. Thus, let $G$ be a locally compact, Hausdorff, second countable,  abelian group.
Let $\Omega$ be an  open subsemigroup of $G$. We assume that $ \Omega$ is spanning, i.e. $\Omega-\Omega=G$ and $0\in\overline{\Omega}$. Suppose $x,y\in G$. We write $x<y$ if $y-x\in\Omega$. 

 Let $(X,\mu)$ be a $\sigma$-finite measure space.
 Suppose $G$ acts on $X$ measurably. We denote the action by $+$. Assume that $\mu$ is $G$-invariant. Let $A$ be a non empty measurable subset of $X$.  Assume also that $A\neq X$. The set $A$  is called an $\Omega$-space if $A+\Omega\subset A$. 
Let $(X,\mu)$ be a $G$-space, and let $A\subset X$ be an $\Omega$-space such that $\displaystyle \bigcap_{b\in\Omega}(A+b)=\emptyset$. 

Define
\begin{equation*}\mathcal{L}_{A}:=\{f:A\to\mathbb{C} :\textrm{ $f$ is Borel, and for each $b\in \Omega$, $ \int_{A\backslash A+b}|f(x)|^{2}dx<\infty$}\}.
\end{equation*} 
As usual, we identify two functions if they are equal almost everywhere. 
 Let $b\in\Omega$. Define a seminorm $||.||_{b}$ on $\mathcal{L}_{A}$ by
\begin{equation*}||f||_{b}:= \Big(\int_{A\backslash A+b} |f(x)|^{2}dx\Big)^{\frac{1}{2}}.
\end{equation*}
The collection of  seminorms $\{||.||_{b}\}_{b\in\Omega}$ defines a locally convex topology  $\tau$ on $\mathcal{L}_{A}$, and it is not difficult to prove that $\mathcal{L}_A$ is a Frechet space. 

\begin{definition}Let $u :\Omega\to\mathcal{L}_{A}$ be a map. We call $u$ a $1$-cocycle if
\begin{enumerate}[(1)]
\item for $a,b,\in \Omega$,
 $ u_{a+b}(x)=u_{a}(x)+u_{b}(x+a)$
  for almost every $x\in A$, and
  \item for $\xi\in\mathcal{L}_{A}$, the map $\Omega\ni b\to \int_{A\setminus A+c} u_{b}(x)\overline{\xi(x)}dx$ is Borel measurable for all $c\in\Omega$.

\end{enumerate}
\end{definition}

 We denote the set of all $1$-cocycles by $Z^{1}(\Omega,\mathcal{L}_{A})$.  We say that $u $ is a \emph{coboundary} if there exists  $f\in \mathcal{L}_{A}$ such that  for every $a\in \Omega$,
\begin{equation*}u_{a}(x)=f(x+a)-f(x)
\end{equation*} for almost every $x\in A$.
Denote the set of coboundaries by $B^{1}(\Omega,\mathcal{L}_{A})$.  Note that  $Z^{1}(\Omega ,\mathcal{L}_{A})$ is a vector space and  $B^{1}(\Omega ,\mathcal{L}_{A})$ is a subspace. The quotient space $\frac{Z^{1}(\Omega,\mathcal{L}_{A})}{B^{1}(\Omega,\mathcal{L}_{A})}$ is denoted by $H^{1}(\Omega,\mathcal{L}_{A})$.  

\begin{remark}
Let $P$ be a closed, convex, pointed and a spanning cone in $\mathbb{R}^d$, and let $\Omega$ be its interior. 
Let $H$ be a closed subgroup of $\mathbb{R}^d$. 
Note that $\mathbb{R}^d$ acts on the homogeneous space $\mathbb{R}^d/H$ by translations. Let $A \subset \mathbb{R}^d/H$ be a non-empty, proper, Borel subset such that $A+P \subset A$. 
Consider the Hilbert space $L^2(A)$. Let $V^{A}:=V=\{V_a\}_{a \in P}$ be the isometric representation  defined by the equation 
\begin{equation}
 V_{a}(f)(x):=\begin{cases}
 f(x-a)  & \mbox{ if
} x-a \in A,\cr
   &\cr
    0 &  \mbox{ if } x-a \notin A.
         \end{cases}
\end{equation}

It is clear that the cohomology group $H^1(\Omega,\mathcal{L}_V)$ introduced in Section 3 coincides with $H^1(\Omega,\mathcal{L}_A)$. The main aim of this section is  compute $H^1(\Omega,\mathcal{L}_A)$ and $H^1_{a}(\Omega,\mathcal{L}_A)$.
\end{remark}

\begin{remark}
 First, we derive the general form of a $1$-cocycle. In the group situation, the general form of a $1$-cocycle for a transtive action is well known. Up to the authors' knowledge, the semigroup analogue 
 does not seem to be available in the literature that serves our purpose. Hence, we have taken some care to include complete proofs in the semigroup situation. 

\end{remark}

Let $H$ be a closed subgroup of $G$. Consider the action of $G$ on $\frac{G}{H}$ by translations. Suppose $\Omega$ is an open, spanning subsemigroup of $G$ such that $0 \in \overline{\Omega}$. Let $A\subset \frac{G}{H}$ be an $\Omega$-space. Note that $\bigcap_{b\in\Omega} (A+b)=\emptyset$, since $A\neq \frac{G}{H}$. Let $\pi:G\to\frac{G}{H}$ be the quotient map. The subgroup $H$ and the $\Omega$-space $A$
is fixed until further mention.

\begin{remark}
\label{no problem with open or closed}
Before proceeding further, we state the following facts.
\begin{enumerate}
\item The set $A\setminus Int(A)$ has measure zero. The reader is referred to Lemma 4.1 of \cite{renault+sundar} for a proof when $A$ is closed. The same  proof works  when $A$ is not necessarily closed.
\item The set $\overline{A}\setminus A$ has measure zero. Observe  that $\overline{A}\setminus A\subset \overline{A}\setminus Int(A)$. Note that $A+\Omega \subset Int(A)$.  Since $0 \in \overline{\Omega}$, $\overline{A}=\overline{Int(A)}$. Thus,  we can  replace $A$ by $Int(A)$, and assume $A$ is open. Note that $A^{c}:=G\setminus A$ is a closed $-\Omega$ space. Observe that $\overline{A}\setminus A=\partial A $. By $(1)$, $A^{c}\setminus Int(A^{c})=\partial A$ has measure zero. Here, $\partial(A)$ denotes the boundary of $A$.
\end{enumerate}
    \end{remark}

\begin{lemma}\label{lemma1}   Let $f: A\to\mathbb{C}$ be Borel. Suppose for any $a\in\Omega$, $f(x+a)=f(x)$ for almost every $x\in A$. Then, there exists $c_{0}\in\mathbb{C}$ such that $f=c_{0}$ almost everywhere.

\end{lemma}
\begin{proof}
Since $A$ is a $\pi(\Omega)$-space, we may assume that $H=0$. In view of Remark \ref{no problem with open or closed}, removing a null set, we can assume that $A$ is open. First, assume $A=\Omega$.
 Let \[N=\{(a,x)\in \Omega\times \Omega : \textrm{$f(x+a)\neq f(x)$}\}.\] Then, $N$ is a Borel subset of $\Omega\times \Omega$. By assumption, the section $N_{a}=\{x\in A : (x,a)\in N\}$ has measure zero for any $a\in\Omega$. By Fubini's theorem, $N$ must have measure zero. 

Apply Fubini's theorem again  to see that for almost every $x\in A$, the set 
\[N^{x}=\{a\in\Omega :\textrm{$(a,x)\in N$}\}\] has measure zero.  Thus, there exists a subset $M$ of $\Omega$ of measure zero  such that for each $x\in\Omega \setminus M$, $f(x+a)=f(x)$ for almost every $a$. Now choose a sequence $x_{n}\in\Omega\setminus M$ such that $x_{n}\to 0$ and $x_{n+1}<x_{n}$ for each $n$. Then, $(\Omega+x_{n})\subset(\Omega+x_{n+1})$, and $\Omega=\bigcup_{n\geq 1}(\Omega + x_{n})$. On each of the sets $\Omega+x_{n}$, $f=f(x_{n})$ almost everywhere.
Since $\Omega+x_{n}\subseteq \Omega+x_{n+1}$, $f(x_{n})=f(x_{n+1})$ for each $n$. So, there  exists a constant $c_{0}\in\mathbb{C}$ such that $f(x_{n})=c_{0}$ for each $n\geq 1$. This proves the result when $A=\Omega$.

If $A$ is a general $\Omega$-space, then $A=\bigcup_{x\in A}(\Omega+x)$ since $A$ is open and $0 \in \overline{\Omega}$. Suppose $f=c_{y}$ almost everywhere on $\Omega+y$. For any $x,y\in A$, $(\Omega+x)\cap (\Omega+y)$ is a non-empty open set, hence has positive measure. Therefore,  $c_{x}=c_{y}$ for each $x,y\in A$. Pick any $x \in A$, and set $c_0:=c_x$. Then, $f=c_0$ almost everywhere.
\end{proof}

\begin{prop}\label{H=0} Suppose $H=0$.  Let $u\in Z^{1}(\Omega , \mathcal{L}_{A})$ be given. Then, there exists a Borel measurable function $f:A\to\mathbb{C}$ such that for each $a\in\Omega$,
\begin{equation*}u_{a}(x)=f(x+a)-f(x)
\end{equation*} for almost every $x\in A$. 
\end{prop}

\begin{proof}
We can assume that $A$ is open. 

\textbf{Case 1:} Suppose $A=\Omega$.
The proof in this case is an adaptation of a similar result from the one parameter case. We include some details of the proof here. The reader may refer to the proof of Thm. 5.3.2 of \cite{arveson} for more details in the $1$-parameter situation.
Suppose $u\in Z^{1}(\Omega,\mathcal{L}_{A})$. Let $v:\Omega\times\Omega\to\mathbb{C}$ be a measurable function such that for almost every $a\in\Omega$,
\begin{equation*}
v(a,x)=u_{a}(x)
\end{equation*} for almost every $x\in A$. By Fubini's theorem, for almost every $x\in\Omega$,
\begin{equation*}v(a+b,x)=v(a,x)+v(b,x+a)
\end{equation*} for almost every $(a,b)\in\Omega\times\Omega$. Suppose $N\subset\Omega$ is of measure zero such that for any $x\in\Omega\setminus N$,
\begin{equation*}v(a+b,x)=v(a,x)+v(b,x+a)
\end{equation*} for almost all $(a,b)\in\Omega\times \Omega$. 
Let $\{a_{n}\}_{n\in\mathbb{N}}$ be a sequence in $\Omega\setminus N$ such that $a_{n}>a_{n+1}$ for each $n\geq 1$, and $\{a_{n}\}$ converges to zero. Then, $\Omega=\cup_{n\geq 1}\Omega+a_{n}$.
For each $n\geq 1$, define a measurable function $f_{n}:\Omega+a_{n}\to\mathbb{C}$ by
\begin{equation*}
    f_{n}(x)=v(x-a_{n}, a_{n}).
\end{equation*} for $x\in \Omega+a_{n}$.
As in Thm. 5.3.2 of \cite{arveson}, the sequence $(f_{n})$ can be corrected and patched together to get a measurable function $f:\Omega\to\mathbb{C}$ such that
\begin{equation*}
    u_{a}(x)=f(x+a)-f(x)
\end{equation*} for almost every $x\in\Omega$.

\textbf{Case 2:} Suppose $A\subset G$ is a general $\Omega$-space. Observe that $A=\bigcup_{x\in A }\Omega+x$, since $A$ is open and $0 \in \overline{\Omega}$. Thanks to the second countability of $A$, there exists a sequence $\{x_{n}\}_{n\in\mathbb{N}}$ in $A$ such that  $A=\bigcup_{n\geq 1}(\Omega+ x_{n})$.
 Since $\Omega+x_{n}$ is a translate of $\Omega$,  for each $n\in\mathbb{N}$, by Case 1, there exists a Borel  function $f_{n}:\Omega+x_{n}\to\mathbb{C}$ such that for each $a\in\Omega$, \begin{equation*}
 u_{a}(x)=f_{n}(x+a)-f_{n}(x)
 \end{equation*} for almost every $x\in \Omega+x_{n}$. Fix $m,n\in\mathbb{N}$ and  assume that $m<n$. Then, the set $B_{m,n}:=(\Omega+x_{n})\cap (\Omega+x_{m})$ is a non empty open $\Omega$-space. Note that for every $a \in \Omega$, \[f_{m}(x+a)-f_{m}(x)=f_{n}(x+a)-f_{n}(x)\] for almost all $x \in B_{m,n}$. 
   Thanks to  Lemma \ref{lemma1}, there exists a complex number $C_{m,n}$ such that $f_m-f_n=C_{m,n}$ almost everywhere on $B_{m,n}$.

Considering the  equality
\begin{equation*}
    f_{m}-f_{n}=\sum_{k=m}^{n-1}(f_{k}-f_{k+1})
\end{equation*}
on  the non-empty open set $\displaystyle D=\bigcap_{k=m}^{n}(\Omega+x_{k})$, we see that 
$
    C_{m,n}=\sum_{k=m}^{n-1}C_{k,k+1}$.
Define a measurable function $g_{n}:\Omega+x_{n}\to\mathbb{C}$ by
\begin{equation*}
    g_{n}=f_{n}+\sum_{k=1}^{n-1}C_{k,k+1}.
\end{equation*}
Finally, define a function $f:A\to\mathbb{C}$ by
\begin{equation*}
 f(x):= \begin{cases}  
    g_{n}(x) &\textrm{if $x\in\Omega+x_{n}\setminus (\bigcup_{k=1}^{n-1}(\Omega+x_k)$)}.
    \end{cases}
\end{equation*} 
Then, $f$ is  well defined, and is  Borel. Also, for $a\in\Omega$, $u_{a}(x)=f(x+a)-f(x)$  for almost every $x\in A$.
\end{proof}

\begin{prop}\label{general form of a one cocycle}
Let $H$ be a closed subgroup of $G$. Let $A\subset\frac{G}{H}$ be an $\Omega$-space. Suppose $u\in Z^{1}(\Omega ,\mathcal{L}_{A})$. Then, there exists a measurable function $f:A\to\mathbb{C}$ and a measurable homomorphism $\sigma : G\to\mathbb{C}$ such that, given $a\in\Omega$, 
\begin{align*}
    u_{a}(x)=f(x+a)-f(x)+\sigma(a)
\end{align*} for almost every $x\in A$.
\end{prop}

\begin{lemma}\label{lemma3} Let $A\subset\frac{G}{H}$ be an $\Omega$-space. Let  $u\in Z^{1}(\Omega,\mathcal{L}_{A})$. Suppose $h\in H$ is of the form $h=a-b$ for some $a,b\in\Omega$. 
Then, there exists $\sigma(h)\in\mathbb{C}$ such that for almost all $x \in A$, $u_{a}(x)-u_{b}(x)=\sigma(h)$. Moreover, the map $\sigma:H \to \mathbb{C}$ is a well defined Borel homomorphism.
\end{lemma}
\begin{proof}
 For $c\in\Omega$, note that
\begin{align*}
   u_{a}(x+c)-u_{b}(x+c) =&u_{a+c}(x)-u_{b+c}(x)\ \ \textrm{(by the cocycle identity)}\\
                         =&(u_{a}(x)+u_{c}(x+a))-(u_{b}(x)+u_{c}(x+b))\\
                         =& u_{a}(x)-u_{b}(x)\ \  \textrm{( since $x+a=x+b$   )}
\end{align*} for almost all $x\in A$. 
By Lemma \ref{lemma1}, there exists a constant $c_{a,b}\in\mathbb{C}$ such that for almost every $x\in A$, $(u_{a}-u_{b})(x)=c_{a,b}$ .
Assume that $h=a-b=a^{'}-b^{'}$ for $a,b,a^{'},b^{'}\in\Omega$. We show that $c_{a,b}=c_{a^{'},b^{'}}$. Observe that
\begin{align*}
u_{a^{'}}(x+a)-u_{b^{'}}(x+a)=u_{a^{'}}(x)-u_{b^{'}}(x)
\end{align*} for almost every $x\in A$.
Note that
\begin{align*} u_{a+b^{'}}(x)+u_{a^{'}}(x+a)-u_{b^{'}}(x+a)=&u_{a^{'}+b}(x)+u_{a^{'}}(x)-u_{b^{'}}(x) ~~( \textrm{since $a+b^{'}=a^{'}+b$})
\end{align*}for almost every $x\in A$. This implies
\begin{align*}u_{a+b^{'}}(x)+u_{b^{'}}(x)-u_{b^{'}}(x+a)=&u_{a^{'}+b}(x)-u_{a^{'}}(x+b)+u_{a^{'}}(x)   ~~\textrm{(since $x+a=x+b$)}
\end{align*} for almost every $x\in A$.
By the cocycle identity, $u_{a+b^{'}}(x)-u_{b^{'}}(x+a)=u_{a}(x)$ and $u_{a^{'}+b}(x)-u_{a^{'}}(x+b)=u_{b}(x)$ for almost every $x\in A$. Now the conclusion follows.
    
Define $\sigma(h)=c_{a,b}$. We have shown that $H\ni h\to\sigma(h)$ is well defined.
It is a routine to check that $\sigma$ is a homomorphism.

Let $g\in G$. Then, there exist measurable maps $\delta_{i}:G\to\Omega$ for $i=1,2$ such that $g=\delta_{1}(g)-\delta_{2}(g)$ (page 361, \cite{Laca}). For $h\in H$,
$\sigma(h)=u_{\delta_{1}(h)}(x)-u_{\delta_{2}(h)}(x)$ for almost every $x\in A$.
Since the map $\Omega\ni a\to u_{a} \in \mathcal{L}_A$ is measurable, $\sigma$ is Borel.
\end{proof}

\begin{proof}[Proof of Prop. \ref{general form of a one cocycle}:] 
 Let $u\in Z^{1}(\Omega,\mathcal{L}_{A})$ be given.
Let $\sigma : H\to\mathbb{C}$ be the Borel homomorphism defined in \ref{lemma3}. Since any measurable homomorphism is continuous, $\sigma$ is continuous. The homomorphism $\sigma$ can be extended to a continuous homomorphism $\widetilde{\sigma}$ on $G$ (see Thm. 3.2 of \cite{Moskowitz}), i.e. there exists a continuous homomorphism $\widetilde{\sigma}:G\to \mathbb{C}$ such that 
\begin{equation*} 
\widetilde{\sigma}(h)=\sigma(h)
\end{equation*} 
for every $h\in H$. 

Let $\widetilde{u}:\Omega\to\mathcal{L}_{A}$ be defined by $\widetilde{u}_{a}=u_{a}-\widetilde{\sigma}(a)$ for $a\in\Omega$, and let  $w:\pi(\Omega)\to\mathcal{L}_{A}$ be defined by $w_{\pi(a)}:= \widetilde{u}_{a}$. It follows from the definition of $\sigma$ that  $w$ is well defined. Clearly, $w$ satisfies the cocycle identity. 
Note that $\widetilde{u}=w\circ\pi$. Hence, $w$ is a Borel (in fact, continuous) map from $\pi(\Omega)$ into $\mathcal{L}_{A}$.

Observe that $\pi(\Omega)-\pi(\Omega)=\frac{G}{H}$, and $A$ is a $\pi(\Omega)$-space. Since $w\in Z^{1}(\pi(\Omega), \mathcal{L}_{A})$, by Proposition \ref{H=0}, there exists a measurable function $f:A\to\mathbb{C}$ such that
\begin{align*}
    w_{\pi(a)}(x)=f(x+a)-f(x)
\end{align*} for almost all $x \in A$. Since, for every $a\in\Omega$, $u_{a}=w_{\pi(a)}+\widetilde{\sigma}(a)$, we have 
\begin{align*}
    u_{a}(x)=f(x+a)-f(x)+\widetilde{\sigma}(a)
\end{align*} for almost every $x\in A$. The proof is complete. 
\end{proof}

\begin{remark} \label{sigma}Let $u\in Z^{1}(\Omega,\mathcal{L}_{A})$. By Proposition \ref{general form of a one cocycle}, there exist a Borel measurable function $f:A\to \mathbb{C}$ and a Borel homomorphism $\sigma:G\to\mathbb{C}$ such that
\begin{align*}
    u_{a}(x)=f(x+a)-f(x)+\sigma(a)
\end{align*} for almost every $x\in A$. We point out here that neither the homomorphism $\sigma$, nor the function $f$ is unique. Suppose $\sigma_{1}:G\to\mathbb{C}$ is a Borel homomorphism such that $\sigma_{1}(h)=\sigma(h)$ for every $h\in H$. Then, there exists a Borel measurable map $g:A\to\mathbb{C}$ such that for $a\in\Omega$,
\begin{align*}
    u_{a}(x)=g(x+a)-g(x)+\sigma_{1}(a)
\end{align*} for almost evrey $x\in A$.
To see this, note that $\sigma-\sigma_{1}$ factors through $\pi$. Let $\widetilde{g}:\frac{G}{H}\to\mathbb{C}$ be such that $\sigma-\sigma_{1}=\widetilde{g}\circ\pi$. 
Now observe that for $a\in\Omega$,
\begin{equation*}
u_{a}(x)=(f+\widetilde{g})(x+a)-(f+\widetilde{g})(x)+\sigma_{1}(a)
\end{equation*}
for almost every $x\in A$. %Let $g=f+\widetilde{g}$. Then, for any $a\in\Omega$,
%\begin{align*}
 %   u_{a}(x)=g(x+a)-g(x)+\sigma_{1}(a)
%\end{align*} for almost every $x\in A$.

Also, suppose  there exist a Borel map $g:A\to\mathbb{C}$ and a measurable homomorphism $\sigma_{1}:G\to\mathbb{C}$ such that for $a\in\Omega$,
\begin{align*}
    u_{a}(x)=g(x+a)-g(x)+\sigma_{1}(a)
\end{align*} for almost every $x\in A$. Then, $\sigma=\sigma_{1}$ on $H$. 
\end{remark}

Hereafter, we consider the case when $G=\mathbb{R}^{d}$. 
Let $P$ be a closed convex cone in $\mathbb{R}^{d}$ that is spanning and pointed. Recall that $\Omega$ stands for  the interior of $P$. Note that $\Omega$ is also a spanning subsemigroup of $\mathbb{R}^{d}$.   Here onwards in this section, all $\Omega$- spaces will be assumed to be closed. 
  The only reason for this is that we have to refer to a few results proved in \cite{piyasa} and in \cite{Anbu}, where only closed $\Omega$-spaces are considered, although the results carry over to any $\Omega$-space. For ease of reference, we consider the closed case. As observed in Remark \ref{no problem with open or closed}, this does not lead to loss of generality. 
  
  Let $H$ be a closed subgroup of $\R^{d}$. Denote the quotient map from $\R^{d}$ onto $\RH$ by $\pi$. For $x\in\R^{d}$, denote $\pi(x)$ by $[x]$ for each $x\in\R^{d}$. Fix an $\Omega$-space $\displaystyle A\subset\frac{\mathbb{R}^{d}}{H}$. Let $\widetilde{A}=\pi^{-1}(A)$. Note that $\widetilde{A}$ is  a closed $\Omega$-space in $\R^{d}$.
  
  Fix a basis $\{v_{i}\}_{i=1}^{d}$ for the vector space $\R^{d}$ such that $v_{i}\in\Omega$ for every $i=1,2,\cdots , d$. We identify $span \ \{ v_{1},v_{2},\cdots , v_{d-1}\}$ with $\R^{d-1}$ and $span\ \{v_{d}\}$ with $\R$. This way, $\R^{d}$ is identified with $\R^{d-1}\times \R$. We will write an element of $\R^{d}$ as $(y,t)$, where $y\in\R^{d-1}$ and $t\in\R$.  We denote the interval $[0,\infty)$ by $\R^{+}$.
  
  By Lemma 4.3 of \cite{Anbu}, there exists a continuous function $\varphi:\R^{d-1}\to\R$ such that 
 \begin{equation*}
          \widetilde{A}=\{(y,t)\in\R^{d}:\textit{ $t\geq\varphi(y)$}\}.
\end{equation*}
The subgroup $H$, the space $A$, the space $\widetilde{A}$, the function $\varphi$,
and the basis $\{v_1,v_2,\cdots,v_d\}$ are fixed for the rest of this paper.

Note that the boundary $\partial \widetilde{A}=\{(y,t)\in\R^{d}:\textrm{ $t=\varphi(y)$}\}$. Since $\widetilde{A}$ is a saturated set, we have  $\partial\widetilde{A}+h=\partial(\widetilde{A}+h)=\partial\widetilde{A}$ for each $h\in H$. Therefore, 
\begin{equation}\label{phi}
    \varphi(y+h_{1})=\varphi(y)+h_{2}
\end{equation} for every $y\in\R^{d-1}$ and $(h_{1},h_{2})\in H$. 

Define
\begin{equation*}
    \Omega_{0}=\{a_{1}\in\R^{d-1}:\textrm{$(a_{1},a_{2})\in\Omega$ for some $a_{2}\in\R$}\}.
\end{equation*}
Then, $\Omega_{0}$ is an open, convex, spanning cone in $\R^{d-1}$. Observe that $\Omega_{0}$ contains the basis $\{v_{1},v_{2},\cdots ,v_{d-1}\}$ of $\R^{d-1}$.

\begin{remark}\label{bdd}
 Fix $a_{1}\in\R^{d-1}$. Let $\widetilde{\varphi}:\R^{d-1}\to\R$ be the function \begin{equation*}y\to \varphi(y+a_{1})-\varphi(y).\end{equation*} We claim that $\widetilde{\varphi}$ is bounded. 
 
 Since any element in $\mathbb{R}^{d-1}$ can be written as a difference of two elements of $\Omega_{0}$, it suffices to consider the case when $a_1 \in \Omega_0$. Let  $a_{1}\in\Omega_{0}$. Suppose $a_{2}\in\R$ is such that $(a_{1},a_{2})\in\Omega$. Since $(0,t) \in \Omega$ for $t>0$, $(a_1,a_2+t) \in \Omega$ for every $t>0$. Thus, we can choose $a_2>0$ such that $(a_1,a_2) \in \Omega$. 
 
 Since $\widetilde{A}$ is an $\Omega$-space, $(y,\varphi(y))+(a_{1},a_{2})\in\widetilde{A}$  for any $y\in \R^{d-1}$, i.e.
 \begin{equation*}\varphi(y+a_{1})\leq \varphi(y)+a_{2}.\end{equation*} Hence,
 \begin{equation*}\varphi(y+a_{1})-\varphi(y)\leq a_{2}\leq Ma_{2}
 \end{equation*} for any $y\in \R^{d-1}$ and $M\geq 1$. Since
 $\{(0,na_{2})\}_{n\geq 1}$ is a cofinal set, there exists a positive integer $M$ such that $(-a_{1}, Ma_{2})\in\Omega$. Therefore, $(y+a_{1},\varphi(y+a_{1}))+(-a_{1},Ma_{2})\in\widetilde{A}$ for every $y \in \mathbb{R}^{d-1}$. This implies that
 \begin{equation*}
     \varphi(y)\leq \varphi(y+a_{1})+Ma_{2}
 \end{equation*} for each $y\in\R^{d-1}$. Hence, $|\varphi(y+a_1)-\varphi(y)|\leq Ma_2$
 for every $y \in \mathbb{R}^{d-1}$.
 \end{remark}
 
 Let \begin{equation*}
      H_{0}:=\{y\in\R^{d-1} :\textit{ $(y,t)\in H$ for some $t\in\R$}\}.
  \end{equation*}
   \begin{lemma}
The subgroup $H_{0}$ is closed.
\end{lemma}
\begin{proof}
Suppose $(y_{n})$ is a sequence in $H_{0}$ that converges to $y\in\mathbb{R}^{d-1}$. For each $n\geq 1$, let $s_{n}\in\R$  be such that $(y_{n},s_{n})\in H$. By Eq. \ref{phi},  $s_{n}=\varphi(y+y_{n})-\varphi(y)$ for each $n\geq 1$. By the continuity of the map $\varphi$, the sequence $(s_{n})$ converges to some $s\in \mathbb{R}$.  Now $((y_{n},s_{n}))$ converges to $(y,s)$. Since $H$ is closed, $(y,s)\in H$. In particular, $y\in H_{0}$. Hence the proof.
\end{proof}
Let $\eta :\R^{d-1}\to\RO$ be the quotient map. Denote $\eta(y)$ by $\langle y\rangle$ for $y\in\R^{d-1}$.
We will denote the Haar measures on $\RH$, $\RO$ and $\R$ by $\mu$, $\mu_{1}$, and $\mu_{2}$ respectively.

  \begin{lemma}

  \label{lemma4}
  Keep the foregoing notation.
  \begin{enumerate}[(1)]
\item The map $\Delta:\RO\times \R\to \RH$ defined by
\begin{equation*}
    \Delta(\langle y\rangle,t)=[(y,\varphi(y)+t)] 
\end{equation*} is a homeomorphism. 
Moreover, $\Delta(\RO\times\R^{+})=A$, and $\Delta(\RO\times\{0\})=\partial A$.

\item Denote by $\boxplus$ the action of $\R^{d}$ on $\RO\times \R$ induced via the homeomorphism $\Delta$. Then, $\Delta$ is an isomorphism between the quadraples $(\RO\times\R,\mu_{1}\times\mu_{2},\boxplus,\RO\times\R^{+})$ and $(\RH,\mu,+, A)$.

\end{enumerate}

 \end{lemma} 

 \begin{proof}
 
  Clearly, the map $\widetilde{\Delta}:\R^{d-1}\times \R\to \R^{d}$ defined by
\begin{equation*}
  \widetilde{\Delta}(y,t)=(y,\varphi(y)+t) 
\end{equation*} is a homeomorphism.
 We claim that the  map $\xi= \pi\circ\widetilde{\Delta}$ factors through  $\eta\times Id$. Suppose $y_{1},y_{2}\in\R^{d-1}$ are such that $y_{1}=y_{2}+z$ for some $z\in H_{0}$. Then, observe that
 \begin{equation*}\xi(y_{1},t)=[(y_{2}+z,\varphi(y_{2}+z)+t)].
 \end{equation*}
 Choose $s \in \mathbb{R}$ such that $(z,s) \in H$. By Eq. \ref{phi}, $\varphi(y_{2}+z)=\varphi(y_{2})+s$. Now, \begin{equation*}\widetilde{\Delta}(y_{1},t)=\widetilde{\Delta}(y_{2},t)+(z,s),
 \end{equation*} which implies
 \begin{equation*}
 \xi(y_{1},t)=\xi(y_{2},t).
 \end{equation*}
 This proves the claim.
 
 Hence, there  exists a continuous map $\Delta:\RO\times \R\to\RH$ such that $\xi=\Delta\circ(\eta\times Id)$. Note that\begin{equation*} 
 \Delta(\langle y\rangle,t)=[(y,\varphi(y)+t)].
 \end{equation*}
 We claim that $\Delta$ is a homeomorphism. Define  $\Phi: \R^{d}\to \RO\times \R$ by 
 \begin{equation*}\Phi(y,t)= (\langle y\rangle,t-\varphi(y))\end{equation*}
 for $y\in \R^{d-1}$ and $t\in\R$. It can be checked similarly that $\Phi$ factors through $\pi$. Hence, there exists  a continuous map $\Psi:\RH\to\RO\times \R$ such that $\Phi=\Psi\circ\pi$. Notice that \begin{equation*}\Psi([(y,t)])=(\langle y\rangle,t-\varphi(y)).\end{equation*} Clearly, $\Delta$ and $\Psi$ are inverses of each other. This completes the proof.

Recall that $A=\pi(\widetilde{A})$. Hence, $A=\{[(y,t)]:\textit{ $t-\varphi(y)\geq 0.$}\}$. From this, it is immediate that \begin{equation*}A=\Delta(\RO\times \R^{+}).\end{equation*} 
It is clear that $\partial A=\Delta(\RO\times \{0\})$.

 Let $(a_{1},a_{2})\in\R^{d}$, and $(\langle y\rangle, t)\in \RO\times\R$. The action $\boxplus$ is given by 
\begin{equation}\label{action}
(\langle y\rangle,t)\boxplus(a_{1},a_{2}):=(\langle y+a_{1}\rangle,\varphi(y)-\varphi(y+a_{1})+a_{2}+t).
\end{equation}
The second assertion follows from the above equation. 
 \end{proof}
Abusing notation, we will henceforth denote the element $\langle y\rangle\in\RO$ by $y$. 
We can now assume that $A$ is the $\Omega$-space $\RO\times\R^{+}$ with the action $\boxplus$. Recall that the formula for the action $\boxplus$ is given by 
\begin{equation}\label{action1}
( y,t)\boxplus(a_{1},a_{2}):=( y+a_{1},\varphi(y)-\varphi(y+a_{1})+a_{2}+t).
\end{equation}

Let us recall the definition of the non-reduced $L^2$-cohomology of a group. Let $G$ be a locally compact, abelian, second countable, Hausdorff group. Let $Z^1(G,L^2(G))$ be the collection of Borel functions $g:G \to \mathbb{C}$ such that 
\[
\int |g(x+z)-g(x)|^{2}dx<\infty
\]
for every $z \in G$. Let $g \in Z^1(G,\mathbb{C})$. We say that $g$ is a coboundary
if $g=c+h$ for a scalar $c$ and for some $h \in L^2(G)$. The space of coboundaries will be denoted by $B^1(G,L^2(G))$. Set \[H^1(G,L^2(G)):=\frac{Z^1(G,L^2(G))}{B^1(G,L^2(G))}.\]
The vector space $H^1(G,L^2(G))$ is called the non-reduced $L^2$-cohomology of $G$.

   Let $g \in Z^1(\RO,L^2(\RO))$ be given.  For $a:=(a_1,a_2) \in \Omega$, define $w^{g}_{a}:A \to \mathbb{C}$ by
     \begin{equation}
        w^{g}_{a}((y,t))=g(y+a_{1})-g(y).
     \end{equation}
     We claim that $w^{g}_{a} \in \mathcal{L}_{A}$. Note that for $n \geq 1$, 
     \[
     \int_{A\setminus (A\boxplus (0,n))}|w_a^{g}(y,t)|^{2}dydt=n \int_{\RO}|g(y+a_{1})-g(y)|^{2} dy<\infty
     \]
     Since $\{(0,n)\}_{n\geq 1}$ is a cofinal set, $w_a^{g} \in \mathcal{L}_A$.
         Clearly, the map $w^g:\Omega \to \mathcal{L}_A$ is a $1$-cocycle. 
  It can be verified  that $w^{g}\in B^{1}(\Omega, \mathcal{L}_{A})$ if and only if $g \in B^{1}(\RO), L^{2}(\RO))$. 

Let us record what we have done above as a proposition.
\begin{prop}
\label{embedding group cohomology}
    The map $Z^1(\RO,L^2(\RO)) \ni g \to w^g \in Z^1(\Omega,\mathcal{L}_A)$ descends to an injective homomorphism 
    \[
    H^1\Big(\RO,L^2\big(\RO\big)\Big)\ni [g] \to [w^g] \in H^1(\Omega,\mathcal{L}_A).
    \]
\end{prop}

     Keep the foregoing notation.
 \begin{lemma}\label{reducing to the boundary} 
 Let $u \in Z^{1}(\Omega , \mathcal{L}_{A})$. Suppose  there exists a Borel map $f:A \to \mathbb{C}$ such that for any $a\in \Omega$, $u_{a}(x)=f(x+a)-f(x)$ for almost all $x \in A$. Then, there exists  $g \in Z^1(\RO,L^2(\RO))$ such that 
       $u$ is cohomologous to the cocycle $w^{g}$.
   \end{lemma}
  \begin{proof} 
   Consider the one parameter cocycle $\widetilde{u}$  defined by $\widetilde{u}_{s}=u_{(0,s)}$ for $s>0$. By Thm. 5.3.2 of \cite{arveson}, there exists  $g_{0}\in\mathcal{L}_{A}$ such that for any $s>0$,
 \begin{equation*}
      \widetilde{u}_{s}(y,t)= g_{0}((y,t)\boxplus (0,s))-g_{0}(y,t) 
  \end{equation*} for almost every $(y,t)\in A$.
  Note that $(y,t)\boxplus (0,s)=(y,s+t)$.
Hence, for every $s>0$,
  \begin{equation*}
      f(y,t+s)-f(y,t)=g_{0}(y,t+s)-g_{0}(y,t)
  \end{equation*}
 for almost every $(y,t)\in A$, i.e. for any $s>0$,
  \begin{equation*}
     (f-g_{0})(y,t+s)=(f-g_{0})(y,t)
 \end{equation*} for almost every $(y,t)\in A$. 
Applying Fubini,
 for almost every $(y,t)\in A$,
  \begin{equation*}
     (f-g_{0})(y,t+s)=(f-g_{0})(y,t)
 \end{equation*} for almost every $s>0$.
 Hence, there exists a Borel measurable function $g:\RO\to\mathbb{C}$ such that
 \begin{equation*}
     (f-g_{0})(y,t)=g(y).
 \end{equation*} for almost every $(y,t)\in A$.
 \item It suffices to show that for any $z\in\R^{d-1}$,
\begin{equation*}
\int_{\RO}\int_{0}^{n}|g(y+z)-g(y)|^{2}dt dy<\infty
\end{equation*}

To that end, it is sufficient to prove that 
 \begin{equation*}
\int_{\RO}\int_{0}^{n}|g(y+a_{1})-g(y)|^{2}dt dy<\infty
\end{equation*} for every $a_{1}\in\Omega_{0}$, since any $z\in\R^{d-1}$ can be written as a difference of two  elements in $\Omega_{0}$.
Let $a_{1}\in\Omega_{0}$. Let $a_{2}\in\R$ be such that $(a_{1},a_{2})\in\Omega$.
Then,
\begin{align*}
&\int_{\RO}\int_{0}^{n}|g(y+a_{1})-g(y)|^{2}dt dy\\
&=\int_{A\setminus A\boxplus (0,n)}|(f-g_{0})(y,t)\boxplus (a_{1},a_{2}))-(f-g_{0})(y,t)|^{2}dydt<\infty .
\end{align*} %Now the proof follows.

  Since $g_{0}\in \mathcal{L}_{A}$,  $u$ and $w^{g}$ are cohomologous. The proof is complete. 
  \end{proof}
  
  We denote the linear span of $H$ by $L(H)$. We denote the codimension of $L(H)$  in $\R^{d}$ by $codim(H)$. As $A$ is assumed to be a proper subset of $\frac{\mathbb{R}^{d}}{H}$, $codim(H)>0$. 

\begin{remark}
\label{measure of lshape}
  Let $B\subset \frac{\mathbb{R}^{d}}{H} $ be an $\Omega$-space. Let us  recall a couple of facts from \cite{piyasa}.
 \begin{enumerate}
     \item For any $a\in\Omega$, $B\setminus B+a$ has finite measure if and only if  $H$ has codimension $1$.
     \item  If $codim(H)=1$, then the boundary of $B$, $\partial B$ is compact.
     
 \end{enumerate}We refer the reader to Proposition 3.4 and Theorem 4.5 of \cite{piyasa} for a proof of the above statements.

 An alternative way to see the above facts is given below. 
 
 First, observe that for  a real number $r$, $(0,r) \in H$ if and only $r=0$. This follows from Eq. \ref{phi}. Consequently, the map
 \[
 H \ni (h_1,h_2) \to h_1 \in H_0 \]
 is an isomorphism. Hence, $span(H)$ and $span(H_0)$ have the same dimension. 
 
 We can assume the action to be $\boxplus$ and $B=\RO \times [0,\infty)$. Since $\{(0,n):n \geq 1\}$ is cofinal, observe that $B\setminus(B\boxplus(a_1,a_2))$ has finite measure for every $(a_1,a_2) \in \Omega$ if and only if $B \setminus (B \boxplus (0,n))=\RO \times [0,n)$ has finite measure for every $n$. Consequently, $B\setminus(B\boxplus(a_1,a_2))$ has finite measure for every $(a_1,a_2) \in \Omega$ if and only if $\RO$ has finite measure, i.e. if $\RO$ is compact. This happens if and only if the linear span of $H_0$ is $\R^{d-1}$ and the latter happens if and only if $H$ has codimension one. 

We already saw that the boundary of $B$ can be identified with $\RO \times \{0\}$. Hence, the boundary of $B$ is compact if and only if $\RO$ is compact if and only if the linear span of $H_0$ is $\mathbb{R}^{d-1}$ if and only if $H$ has codimension one.

 \end{remark}
\begin{remark}\label{zero cohomology for compact groups} Suppose $H$ has codimension 1.  The cocycle $w^{g}$ is a coboundary for every $g \in Z^1(\RO,L^2(\RO))$. This is because if $H$ has codimension one, then $\RO$ is a compact group, and for a compact group $G$, $H^1(G,L^2(G))=0$ (see Section 1.4 of \cite{Guichardet}).
\end{remark}

Let $F(A,\mathbb{C})$ denote the set of all Borel maps from $A$ into $\mathbb{C}$. 
  For $\lambda_1,\lambda_2 \in \bbr^d$, define the map $ u^{\lambda_{1},\lambda_{2}}:\Omega\to F(A,\mathbb{C}) $ by
 \begin{equation*}
     u^{\lambda_{1},\lambda_{2}}_{a}(x):=\langle\lambda_{1}|a\rangle+i\langle\lambda_{2}|a\rangle.
 \end{equation*} 
\begin{theorem}\label{computation of cohomology group}
With the foregoing notation, we have the following. 
    \begin{enumerate}[(1)]
    \item Suppose the subgroup $H$ has codimension one. For $\lambda_{1},\lambda_{2}\in L(H)$, $u^{\lambda_{1},\lambda_{2}}\in Z^{1}(\Omega,\mathcal{L}_{A})$. Let $u\in Z^{1}(\Omega, \mathcal{L}_{A})$. Then, there exist $\lambda_{1},\lambda_{2}\in L(H)$ such that $u$ is cohomologous to $u^{\lambda_{1},\lambda_{2}}$.
       Moreover, the map
       \begin{equation*}
         (\lambda_{1},\lambda_{2}) \ni L(H)\oplus L(H) \to [u^{\lambda_{1},\lambda_{2}}] \in H^1(\Omega,\mathcal{L}_A)
     \end{equation*}
  is an isomorphism. Here, $L(H) \oplus L(H)$ is considered as the complexification of $L(H)$.
    \item Suppose $codim(H)\geq 2$.  Then, the map 
    \[
    H^1\Big(\RO,L^2\big(\RO\big)\Big)\ni [g] \to [w^g] \in H^1(\Omega,\mathcal{L}_A)\]
    is an isomorphism. 
    \end{enumerate}
    \end{theorem}
    \begin{proof}
First consider the case when $codim(H)=1$.      Recall that $A\setminus A\boxplus a$ has finite measure for every $a\in \Omega$. So, for any  $\lambda_{1},\lambda_{2}\in L(H)$, $u^{\lambda_{1},\lambda_{2}}_{a}\in\mathcal{L}_{A}$. It is also clear that for $\lambda_1,\lambda_2 \in L(H)$,  the map $ u^{\lambda_{1},\lambda_{2}}$ is a cocycle.
 
 Let $u\in Z^{1}(\Omega,\mathcal{L}_{A})$. By Prop. \ref{general form of a one cocycle}, there exists a Borel map $f:A\to\mathbb{C}$ and a Borel homomorphism $\sigma:\R^{d}\to\mathbb{C}$ such that for $a\in\Omega$,
 \begin{equation*}
     u_{a}(x)=f(x+a)-f(x)+\sigma(a)
 \end{equation*} for almost every $x\in A$.
 Choose $\lambda_1,\lambda_2 \in L(H)$ such that $\sigma(a)=\langle \lambda_1|a \rangle + i \langle \lambda_2|a \rangle$ for $a \in H$. 
  By Remark \ref{sigma},  there exists a Borel map $\widetilde{f}:A\to \mathbb{C}$ such that for $a\in\Omega$,
 \begin{equation*}
   u_{a}(x)=  \widetilde{f}(x+a)-\widetilde{f}(x)+\langle \lambda_1|a \rangle+i\langle \lambda_2|a \rangle.
 \end{equation*} for almost every $x\in A$. 
 Thanks to Lemma \ref{reducing to the boundary}, $u-u^{\lambda_1,\lambda_2}$ is cohomologous to $w^g$ for some $g \in Z^1(\RO, L^2(\RO))$. By Remark \ref{zero cohomology for compact groups}, $w^g$ is a coboundary. Therefore, $u$ is cohomologous to $u^{\lambda_1,\lambda_2}$. This shows that the map
 \[
 L(H) \oplus L(H) \ni (\lambda_1,\lambda_2) \to [u^{\lambda_1,\lambda_2}] \in H^1(\Omega,\mathcal{L}_A)
 \]
 is onto. 
    
Suppose there exist $\lambda_{1},\lambda_{2} \in L(H)$  such that $u^{\lambda_{1},\lambda_{2}}$ is a coboundary. Then, there exists $f \in \mathcal{L}_A$ such that for $a \in \Omega$,
\[
\sigma(a)=f(x+a)-f(x)
\]
for almost all $x \in A$. Here, $\sigma(a)=\langle \lambda_1|a \rangle+i\langle \lambda_2|a \rangle$. 

Let $h \in H$ be given. Write $h=a-b$ with $a,b \in \Omega$ and calculate as follows to observe that for almost all $x \in A$,
\[
\sigma(a)=f(x+a)-f(x)=f(x+b)-f(x)=\sigma(b).
\]
Hence, $\sigma(h)=0$ for every $h \in H$. This forces that $\lambda_1=\lambda_2=0$. 
Consequently, the map 
\[
 L(H) \oplus L(H) \ni (\lambda_1,\lambda_2) \to [u^{\lambda_1,\lambda_2}] \in H^1(\Omega,\mathcal{L}_A)
 \]
 is one-one. This completes the proof of $(1)$.

    Now suppose $codim(H)\geq 2$. We have already seen in Prop. \ref{embedding group cohomology} that the map 
    \[
    H^1\Big(\RO,L^2\big(\RO\big)\Big) \ni [g] \to [w^g] \in H^1(\Omega,\mathcal{L}_A)
    \]
    is injective. 

    Let $u \in Z^1(\Omega,\mathcal{L}_A)$ be given. By Prop. \ref{general form of a one cocycle}, there exists a Borel map $f:A \to \mathbb{C}$ and a Borel homomorphism $\sigma:\mathbb{R}^d \to \mathbb{C}$ such that for $a \in \Omega$,
    \[
    u_a(x)=\sigma(a)+f(x+a)-f(x)
    \]
    for almost all $x \in A$.     Observe that for $h\in H$, $\sigma(h)=u_{a}(x)-u_{b}(x)$ for almost every $x\in A$, whenever $h=a-b$ for $a,b\in \Omega$. So the constant function $\sigma(h)$ must be integrable on $A\setminus A+ d$ for every $d\in\Omega$. But when $codim(H)\geq 2$, by Remark \ref{measure of lshape}, there exists $d \in \Omega$, such that $A\setminus (A+d)$ has infinite measure. So, $\sigma|_{H}=0$. Again, by Remark \ref{sigma}, we may assume $\sigma$ is identically zero on $\R^{d}$. Hence for $a\in\Omega$, 
 \begin{equation*}
   u_{a}(x)=  f(x+a)-f(x)
 \end{equation*} for almost every $x\in A$. By Lemma \ref{reducing to the boundary}, there exists $g \in Z^1(\RO,L^2(\RO))$ such that   $u$ is cohomologous to $w^{g}$. This completes the proof of the surjectivity of the map 
 \[
    H^1\Big(\RO,L^2\big(\RO\big)\Big) \ni [g] \to [w^g] \in H^1(\Omega,\mathcal{L}_A).
    \]
    The proof is complete.     \end{proof}
 
%\subsection{Computation of $H_{ad}^{1}(\Omega,\mathcal{L}_{A})$} 
 Let $A=\RO\times \R^{+}$ with the action $\boxplus$. Let $V:=V^A=\{V_a\}_{a \in P}$ be the shift semigroup associated to $A$.
For $u\in Z^{1}(\Omega,\mathcal{L}_{A})$, we denote the extension of $u$ to $P$
again by $u$. Let $\Gamma^u$ be the associated $2$-cocycle considered in Section 3. 
Define $T^{u}:P \times P\times P\to\R$ by
 \begin{equation}\label{T}
  T^{u}(a,b,c):=  Im\langle\Gamma^{u}(a,b+c)|V_{b}\Gamma^{u}(b,c)\rangle.
 \end{equation} 
 Recall that the cocycle $u$ is admissible if and only if there exists a measurable map $\alpha:P \times P \to\mathbb{T}$
 such that
 \begin{equation}\label{alpha}
    e^{iT^{u}(a,b,c)}=\frac{\alpha(a,b)\alpha(a+b,c)}{\alpha(a,b+c)\alpha(b,c)}
\end{equation}
 for every $a,b,c\in P$.

 The space of admissible cocycles will be denoted $Z^1_{a}(\Omega,\mathcal{L}_A)$ and its image in $H^1(\Omega,\mathcal{L}_A)$ will be denoted by $H^1_{a}(\Omega,\mathcal{L}_A)$.
 We will denote $\Gamma^{u}$ by $\Gamma$,  and $T^{u}$ by $T$ when the cocycle $u$ is clear. 
 
We will address the question of admissibility  of cocycles  when $codim(H)=1$ and $codim(H)\geq 2$ separately.

\begin{remark}
\label{anstisymmetric trick}
Let $X$ be a non empty set, and let  $F:X\times X\times X\to \R$ be a map. Recall that the anti-symmetric part of $F$, denoted by $F_{*}$, is defined by
\begin{align*}
F_{*}(x_{1},x_{2},x_{3})=&\frac{1}{6}\sum_{\sigma\in S_{3}}sgn(\sigma) F(x_{\sigma(1)},x_{\sigma(2)},x_{\sigma(3)})
\end{align*} for $(x_{1},x_{2},x_{3})\in X\times X\times X$.

 Suppose $S$ is an abelian semigroup. Let $F:S\times S\times S\to \R$ be a map. Suppose there exists a function $\beta :S\times S\to \R $ such that for any  $(a,b,c)\in S\times S \times S$, \begin{equation*}F(a,b,c)=\beta(a,b)+\beta(a+b,c)-\beta(a,b+c)-\beta(b,c).
 \end{equation*}
 Then, $F_{*}=0$.
 The proof is a straightforward computation. 
 \end{remark}

 \begin{theorem}
 \label{H1 codim one}
 Suppose $codim(H)=1$. Let $\gamma_{1},\gamma_{2} \in L(H)$. The cocycle $u^{\gamma_{1},\gamma_{2}}$ is admissible if and only if $\gamma_{1}$ and $\gamma_{2}$ are linearly dependent.    
    Therefore,
    \begin{equation*}
    H_{a}^{1}(\Omega,\mathcal{L}_{A})=\{[u^{\gamma_{1},\gamma_{2}}]:\textrm{ $\gamma_{1},\gamma_{2}\in L(H)$, $\gamma_{1},\gamma_{2}$ are linearly dependent}\}.
    \end{equation*}
    \end{theorem}
\begin{proof}
Let $\gamma_1,\gamma_2 \in L(H)$ be given. Let $\Gamma$ be the $2$-cocycle that corresponds to $u^{\gamma_1,\gamma_2}$. Then, for $a,b \in P$, 
  $\Gamma(a,b)=(\langle\gamma_{1}|a\rangle+i\langle \gamma_{2}|a\rangle)1_{A\setminus A\boxplus b}$.
  Define $T:P \times P \times P \to \mathbb{R}$ by 
 \[  T(a,b,c):=  Im\langle\Gamma(a,b+c)|V_{b}\Gamma(b,c)\rangle.\]
 If $\gamma_1$ and $\gamma_2$ are linearly dependent, it is clear that $T=0$. Consequently, $u^{\gamma_1,\gamma_2}$ is admissible. 

 Conversely, assume that $u^{\gamma_1,\gamma_2}$ is admissible. Note that  
\begin{equation*}
  T(a,b,c)
  =Im \int_{A\boxplus b\setminus A\boxplus (b+ c)}(\langle\gamma_{1}|a\rangle \langle\gamma_{2}|b\rangle-\langle\gamma_{1}|b\rangle\langle\gamma_{2}|a\rangle) dx
\end{equation*}

Define $\rho:\Omega\to\mathbb{R}$ by \begin{equation*}\rho(a)=\langle 1_{A\setminus A\boxplus a}|1_{A\setminus A\boxplus a}\rangle.\end{equation*}  Observe that $\rho$ is a non-zero Borel semigroup homomorphism.
So, $\rho$ extends to a measurable homomorphism from $\mathbb{R}^{d}$ to $\R$. Denote the extension by $\rho$ as well. Note that $\rho:\R^{d}\to\R$ is  continuous, hence $\rho$ is a linear functional. Let $\gamma_{3}\in\R^{d}$ be such that $\rho(a)=\langle\gamma_{3}|a\rangle$ for each $a\in\R^{d}$.
We claim that $\gamma_{3}\in  L(H)^{\perp}$. Let $h\in H$. Suppose $h=a-b$ for $a,b\in\Omega$. Then, $\rho(a)=\rho(b)$ by the definition of $\rho$. So, $\langle\gamma_{3}|h\rangle=0$. Hence the claim is proved.

Notice that
\begin{equation*}
T(a,b,c)=\langle\gamma_{3}|c\rangle(\langle\gamma_{1}|a\rangle \langle\gamma_{2}|b\rangle-\langle\gamma_{1}|b\rangle\langle\gamma_{2}|a\rangle) .
\end{equation*} for $a,b,c\in\Omega$.

Suppose there exists a Borel map $\alpha: P \times P \times P \to\mathbb{T}$ such that
 \begin{equation}\label{eqn1}
\frac{\alpha(a,b)\alpha(a+b,c)}{\alpha(a,b+c)\alpha(b,c)}=e^{i T(a,b,c)}
\end{equation} for $a,b,c\in P$. Let $\beta:P \times P\to\R$ be a map such that 
\begin{equation*}
e^{i\beta(a,b)}=\alpha(a,b)
\end{equation*}for each $(a,b)\in P\times P$. Then, there exists a map $K: P \times P \times P \to\mathbb{Z}$ such that
\begin{align}\label{eq5} 
T(a,b,c)=\beta(a,b)+\beta(a+b,c)-
\beta(a,b+c)-\beta(b,c)+2\pi K(a,b,c)
\end{align} for each $(a,b,c)\in P \times P \times P$.

 Let $W:P \times P \times P \to\R$ be the map defined \begin{equation*}W(a,b,c)=\beta(a,b)+\beta(a+b,c)-
\beta(a,b+c)-\beta(b,c).\end{equation*}
By Remark \ref{anstisymmetric trick}, the anti-symmetric part of $W$, i.e. $W_{*}=0$.
Hence, for $a,b,c \in P$, $T_{*}(a,b,c)=2\pi K_{*}(a,b,c)\in \frac{\pi}{3}\mathbb{Z}$. 
Observe that
\begin{align}\label{det}
   T_{*}(a,b,c)=\frac{1}{6}Det \begin{pmatrix}
\langle\gamma_{1}|a\rangle & \langle\gamma_{1}|b\rangle & \langle\gamma_{1}|c\rangle\\
\langle\gamma_{2}|a\rangle & \langle\gamma_{2}|b\rangle& \langle\gamma_{2}|c\rangle\\
\langle\gamma_{3}|a\rangle & \langle\gamma_{3}|b\rangle& \langle\gamma_{3}|c\rangle
\end{pmatrix}
\end{align}
Since $T_{*}$ is a continuous function on $P \times P \times P $  which approaches $0$ as $(a,b,c)\to 0$ and $T_*(a,b,c) \in \frac{\pi}{3}\mathbb{Z}$, $T_{*}=0$.  So, $\gamma_{1},\gamma_{2},\gamma_{3}$ are linearly dependent. Since $\gamma_{3}\in L(H)^{\perp}$, and since $\gamma_1,\gamma_2 \in L(H)$, $\gamma_{1}$ and $\gamma_{2}$ are linearly dependent.

Now the equality  \begin{equation*}
    H_{a}^{1}(\Omega,\mathcal{L}_{A})=\{[u^{\gamma_{1},\gamma_{2}}]:\textrm{ $\gamma_{1},\gamma_{2}\in L(H)$, $\gamma_{1},\gamma_{2}$ are linearly dependent}\}
    \end{equation*}
   follows from Thm. \ref{computation of cohomology group} and from Prop. \ref{admissibility of cohomologous cocycles}.

 \end{proof}

\textbf{Notation:} For a measurable function $g$ on $\RO$ and for $z \in \RO$, we 
let $L_zg$ be the function defined by
\[
L_zg(y)=g(y+z).
\]

\begin{theorem}\label{H1 codim greater than 1}
     Suppose $codim(H)\geq 2$. Then, every $u\in Z^{1}(\Omega, \mathcal{L}_{A})$ is admissible, and 
    \begin{equation*}H_{a}^{1}(\Omega,\mathcal{L}_{A})=H^{1}\Big(\RO,L^{2}\big(\RO\big)\Big).
    \end{equation*}
\end{theorem}
\begin{proof}
 Fix $u\in Z^{1}(\Omega,\mathcal{L}_{A})$. By Thm. \ref{computation of cohomology group}, it follows that $u$ is cohomologous to  $w^{g}$ for some $g \in Z^1\Big(\RO,L^2\big(\RO\big)\Big)$.  Thanks to Prop. \ref{admissibility of cohomologous cocycles}, it suffices to show that $w^g$ is admissible.

 It is known that the reduced cohomology $H^{1}_{red}\Big(\RO,L^{2}\big(\RO\big)\Big)=0$ (see  Section 1.4 of \cite{Guichardet}). Therefore, there exists a sequence $(g_{n})\in L^{2}(\RO)$ such that for every $z \in \RO$, $L_zg_n-g_n \to L_zg-g$ in $L^2(\RO)$. Moreover, the convergence is uniform in $z$. 
 Let $\Gamma_n$ be the $2$-cocycle associated to $w^{g_n}$ and let $\Gamma$ be the cocycle associated to $w^{g}$. 

 Define $T:P \times P \times P \to \mathbb{C}$ by 
 \[
 T(a,b,c)=Im \langle \Gamma(a,b+c)|V_b\Gamma(b,c)\rangle,
 \]
 and define $T_n: P \times P \times P \to \mathbb{C}$ by 
 \[
 T_n(a,b,c)=Im \langle \Gamma_n(a,b+c)|V_b\Gamma_n(b,c)\rangle.
 \]
  For $y \in \RO$ and $b=(b_1,b_2) \in \bbr^d$, let 
  \[
  r(y,b):=\varphi(y-b_1)-\varphi(y)+b_2.
  \]
    By a straightforward computation, we observe that for $(a,b,c) \in P \times P \times P$,
 \begin{equation*}
  T(a,b,c):= Im \int_{\RO}\int_{r(y,b)}^{r(y,b+c)}(g(y+a_{1})-g(y))(\overline{g(y)-g(y-b_{1})}dtdy.
 \end{equation*}  
   Similarly,
 \begin{equation*}
  T_{n}(a,b,c):= Im\int_{\RO} \int_{r(y,b)}^{r(y,b+c)}(g_{n}(y+a_{1})-g_{n}(y))\overline{(g_{n}(y)-g_{n}(y-b_{1}))}dtdy
 \end{equation*} whenever $a=(a_{1},a_{2}),$ $b=(b_{1},b_{2}),$ $c=(c_{1},c_{2})\in P$.
 
 The map $\widetilde{T}:\R^{d}\times \R^{d}\times \R^{d}\to \R$ defined by
  \begin{equation*}
  \widetilde{T}(a,b,c):= Im \int_{\RO}\int_{r(y,b)}^{r(y,b+c)}(g(y+a_{1})-g(y))(\overline{g(y)-g(y-b_{1})}dtdy
 \end{equation*} is Borel, and extends $T$ to $\R^{d}\times \R^{d}\times \R^{d}$. Similarly, the maps $\widetilde{T_{n}}$ can be defined for each $n\geq 1$.
 Abusing notation, we denote $\widetilde{T}$ as $T$ and $\widetilde{T_{n}}$ as $T_{n}$  respectively. By Remark \ref{bdd}, given $b,c \in \R^d$, there exists  exists $M>0$ such that $|r(y,b)-r(y,b+c)|\leq M$.
 
Now it is clear that the sequence $T_{n}$ converges to $T$ pointwise on $\R^{d}\times\R^{d}\times\R^{d}$. 
 For $n\in\mathbb{N}$, let $\beta_{n}:\R^{d}\times\R^{d}\to\mathbb{R}$ be defined by \begin{equation*}\beta_n(a,b)=  -Im\int_{\RO}\int_{0}^{r(y,b)}g_{n}(y+a_{1})\overline{g(y)}dtdy.
 \end{equation*}
 
 Then, \begin{equation} \label{Tn1}
T_{n}(a,b,c)=\beta_{n}(a,b)+\beta_{n}(a+b,c)-\beta_{n}(a,b+c)-\beta_{n}(b,c)
 \end{equation} for every $n\geq 1$ and $a,b,c\in\mathbb{R}^d$.
 Hence, $T_{n}\in B^{3}(\R^{d},\mathbb{R})$ for each $n\geq 1$. Since $T_{n}\to T$ pointwise, $T$ is a $3$-cocycle. 
 
 By the Van Est theorem (Corollary 7.2 of \cite{Guichardet2}), $T$ is cohomologous to an anti-symmetric 3-linear form, say $W$.  Eq. \ref{Tn1}, the fact that $T_n \to T$ pointwise  and Remark \ref{anstisymmetric trick} imply that the anti-symmetric part of $T$, i.e. $T_{*}=0$. Note that $W=W_{*}$. Since the anti-symmetric part of a coboundary in $Z^{3}(\R^{d},\R)$ is zero, it follows that $W_{*}=0$, i.e $T$ is a coboundary. Hence, there exists a Borel map
 $\beta:\R^{d}\times\R^{d}\to\R$ such that \[
T(a,b,c)=\beta(a,b)+\beta(a+b,c)-\beta(a,b+c)-\beta(b,c)
 \] for $a,b,c\in P$. Define  $\alpha:P \times P \to \mathbb{T}$  by 
$
 \alpha(a,b):=e^{i\beta(a,b)}.
 $
 Then,
 \[
e^{iT(a,b,c)}=\frac{\alpha(a,b)\alpha(a+b,c)}{\alpha(a,b+c)\alpha(b,c)}
 \] for $a,b,c\in P$, i.e. $u$ is admissible. 
 The equality 
 \begin{equation*}H_{a}^{1}(\Omega,\mathcal{L}_{A})=H^{1}\Big(\RO,L^{2}\big(\RO\big)\Big)
    \end{equation*}
follows from Thm. \ref{computation of cohomology group} and from Prop. \ref{admissibility of cohomologous cocycles}. The proof is complete.
 \end{proof}

We end this paper by parametrising $\mathcal{D}_{V^A}(P)$ in the co-dimension one case. 
Let $V$ be a pure isometric representation of $P$ on a separable Hilbert space $\mathcal{H}$. Recall that $D_{V}(P)$ denotes the collection of all decomposable product systems (upto projective isomorphism) over $P$ whose associated isometric representation is unitarily equivalent to $V$. 
  Denote the commutant $\{V_{a},V_{a}^{*}\}^{'}$ by $M_{V}$. Let $\mathcal{U}(M_{V})$ be the unitary group of $M_{V}$. By Eq. \ref{succinct},  $D_{V}(P)$ is in bijective correspondence with $\frac{H^{1}_{a}(P,\mathcal{L}_V)}{\mathcal{U}(M_{V})}=\frac{H^{2}_{a}(P,\mathcal{H})}{\mathcal{U}(M_V)}$.
  
 Let $H$ be a closed subgroup of $\R^{d}$ for $d\geq 2$. Let $A\subset \RH$ be a $P$-space. Consider the shift semigroup $V:=V^{A}$ associated to $A$.
 We next obtain a neat parametrisation of $\mathcal{D}_{V^A}(P)$ when $codim(H)=1$.
  Let $L(H)$ denote the linear span of $H$. Define an equivalence relation $\sim$ on $L(H)$ by $\lambda\sim\mu$ if and only if $\lambda=\pm\mu$ for $\lambda,\mu\in L(H)$.
 
 \begin{prop}
 \label{parametrisation in codim1}
 Keep the foregoing notation. If $codim(H)=1$, then  $D_{V^{A}}(P)$ is in bijective correspondence with $\frac{L(H)}{\sim}$. 
      \end{prop}
 \begin{proof}We write $V^{A}=V$ for simplicity. Then, from Thm. \ref{H1 codim one}, 
\[H^{1}_{a}(P,\mathcal{L}_{A})\cong\{(\lambda,\mu)\in L(H)\oplus L(H):\textrm{ $\lambda$, $\mu$ are linearly dependent}\}.\] An element $(\lambda,\mu)\in L(H)\oplus L(H)$ such that $\lambda,\mu $ are linearly dependent corresponds to the 1-cocycle $u\in Z^{1}_{a}(P,\mathcal{L}_A)$ given by $u_{a}=\langle\lambda|a\rangle+i\langle\mu|a\rangle$ for $a\in P$. Since $\lambda,\mu$ are linearly dependent, there exist $\theta \in \bbr$ and $\lambda_0 \in L(H)$ such that  $u_a(x)=e^{i\theta}\langle \lambda_0|a \rangle$ for almost all $x$. Since the unitary $L^{2}(A)\ni f\to e^{i\theta}f\in L^{2}(A)$ belongs to $M_{V}$, $[(\lambda,\mu)]=[(\lambda_{0},0)]$ in $\displaystyle \frac{H^{1}_{a}(P,\mathcal{L}_A)}{\mathcal{U}(M_{V})}$.

 Let $\lambda,\mu\in L(H)$ be such that $[(\lambda,0)]=[(\mu,0)]$ in $\displaystyle \frac{H^{1}_{a}(P,\mathcal{L}_A)}{\mathcal{U}(M_V)}$. Then, there exists $U\in M_{V}$  and  $g\in \mathcal{L}_{A}$ such that
 \begin{equation}
 \label{lambda mu cohomologous}V_{a}(\langle\lambda|a\rangle 1_{A\setminus A+b}-\langle\mu|a\rangle U(1_{A\setminus A+b}))=g_{a+b}-g_{a}-V_{a}g_{b}
 \end{equation} for $a,b\in P$, where $g_{a}=g1_{A\setminus A+a}
 $ for $a\in P$. For $a\in P$, define $f_{a}:=U1_{A\setminus A+a}$ . Note that $f_{a}\in Ker(V_{a}^{*})$ for $a\in P$. Moreover, \[f_{a}+V_{a}f_{b}=f_{a+b}\] for $a,b\in P$. By Prop. 3.2 of \cite{piyasa},for each $a\in P$, there exists a scalar $c\in\mathbb{T}$ such that $f_{a}=c1_{A\setminus A+a}$ for $a\in P$ .
 Therefore, by Eq. \ref{lambda mu cohomologous},  for each $a\in P$,\[\langle\lambda|a\rangle+c\langle\mu|a\rangle=g(x+a)-g(x)\] for almost every $x\in A$. 
 
 Consider the 1-cocycle $w\in Z^{1}(P,\mathcal{L}_{A})$ given by \[w_{a}(x)=g(x+a)-g(x)-\langle\lambda|a\rangle+c\langle\mu|a\rangle\] for $a\in P$, $x\in A$. Then, $w=0$. By Remark \ref{sigma}, $\langle\lambda|a\rangle-c\langle\mu|a\rangle=0$ for $a\in L(H)$. Hence, we must have $c\in\R$, i.e. $c=\pm 1$ and $\lambda=c\mu$, i.e $\lambda=\pm \mu$. Therefore, $[(\lambda,0)]=[(\mu,0)]$ if and only if $\lambda=\pm\mu$.  Now, the Theorem follows from Eq. \ref{succinct}.
\end{proof}

\begin{remark}
If $codim(H) \geq 2$, it is difficult to obtain a neat `explicit' parametrisation of $\mathcal{D}_{V^A}(P)$, by a finite dimensional space, as in Prop. \ref{parametrisation in codim1}. By Eq. \ref{succinct}, \[D_{V^A}(P)=  
\displaystyle \frac{H^1(\RO,L^2(\RO))}{\mathcal{U}(M_{V^A})}.\] In this situation, there is no `good basis' or `no good coordinates' that parametrises $D_{V^A}(P)$. In fact, it is possible to find uncountably many classes $\{[u_i]: i \in I\}$  such that if $v_i$ represents $[u_i]$, then $\{v_i: i \in I\}$ is linearly independent. 

To illustrate, consider the example when $H=0$ and $A$ has no stabiliser, i.e.
\[\{z \in \mathbb{R}^{d}: A+z=A\}=\{0\}.\] By Corollary 3.4 of \cite{Anbu_Sundar}, it follows that $V:=V^{A}$ is irreducible, i.e. $\{V_x,V_{x}^{*}:x \in P\}^{'}=\mathbb{C}$. By Eq. \ref{succinct} and by Prop. \ref{H1 codim greater than 1}, it follows that 
\[
\mathcal{D}_{V^A}(P)=\frac{H^1(\mathbb{R}^{d-1},L^{2}(\mathbb{R}^{d-1}))}{\mathbb{T}}.\]
The action of $\mathbb{T}$ on $H^1(\mathbb{R}^{d-1},L^{2}(\mathbb{R}^{d-1}))$ is given by $\lambda.[u]=[\lambda u]$.

It is known that the space $H^1(\mathbb{R}^{d-1},L^{2}(\mathbb{R}^{d-1}))$ is infinite dimensional. We give an explicit example of uncountably many linearly independent elements in $H^1(\bbr^{d-1},L^{2}(\bbr^{d-1}))$ below.

Fix $d\geq 2$. Let $p>0$. Let $f^{p}:\mathbb{R}^{d-1}\to\mathbb{R}$ be the function defined by 
 \[
 f^{p}(x_{1},x_{2},\cdots, x_{d-1})=\begin{cases}\frac{1}{(\sum_{i=1}^{d-1}x_{i})^{p}}, & \textit{if $x_{i}> 1$ for each $i=1,2,\cdots d-1$},\\
 0, &\textit{otherwise.}
 \end{cases}\]
  Observe that $f^{p}\in L^{2}(\R^{d-1})$ if and only if $p>\frac{d-1}{2}$.  For $a\in \R^{d-1}$, let the function $g_{a}:\R^{d-1}\to\R^{d-1}$ be defined by $g_{a}(x)=f^{p}(x+a)-f^{p}(x)$ for $x\in\R^{d-1}$. It can be verified that  $g_{a}\in L^{2}(\R^{d-1})$ if $p>\frac{d-2}{2}$. Let $J=\{p:\frac{d-2}{2}<p<\frac{d-1}{2}\}$. Then,   the set $\{[f^{p}]:p \in J\}$ is linearly independent in $H^1(\bbr^{d-1},L^2(\bbr^{d-1}))$.
 
\end{remark}
 \bibliography{reference}
\bibliographystyle{plain}

\end{document}